\documentclass{article}

\usepackage{latexsym}
\usepackage{amsmath, amssymb}
\usepackage{amsthm}

\newcommand{\eref}[1]{(\ref{#1})}
\newcommand{\cut}{\right.\\\left.}

\newcommand{\R}{{\mathbb{R}}}
\newcommand{\C}{{\mathbb{C}}}
\newcommand{\Q}{{\mathfrak{Q}}}
\newcommand{\rmx}{{\rho_{M}}}
\newcommand{\el}{{\mathcal{L}}}
\newcommand{\dom}{\mathcal{D}}
\newcommand{\K}{\mathcal{K}}
\newcommand{\im}{\mathrm{Im}\,}
\newcommand{\re}{\mathrm{Re}\,}
\newcommand{\oh}{\mathcal{O}}
\newcommand{\op}{\left(-\partial_t^2+\partial_r^2+\frac{1}{r}\partial_r\right)}
\newcommand{\opR}{\partial_R^2+\frac{1}{R}\partial_R}
\newcommand{\opr}{\partial_r^2+\frac{1}{r}\partial_r}

\newcommand{\ept}{\tilde{\epsilon}}

\newcommand{\s}[3]{S^{#1}\left(R^{#2}(\log R)^{#3}\right)}
\newcommand{\sq}[4]{S^{#1}\left(R^{#2}(\log R)^{#3},\Q_{#4}\right)}
\newcommand{\sqp}[4]{S^{#1}\left(R^{#2}(\log R)^{#3},\Q_{#4}'\right)}

\newcommand{\isq}[4]{IS^{#1}\left(R^{#2}(\log R)^{#3},\Q_{#4}\right)}
\newcommand{\isqp}[4]{IS^{#1}\left(R^{#2}(\log R)^{#3},\Q_{#4}'\right)}

\newcommand{\tl}[1]{\frac{1}{(t\lambda)^{#1}}}

\newcommand{\lr}[1]{L_{\rho}^{2,#1}}
\newcommand{\nlr}[2]{\left\lVert #1 \right\rVert_{\lr{#2}}}
\newcommand{\llr}[2]{L^{\infty,#1}\lr{#2}}
\newcommand{\nllr}[3]{\left\lVert #1 \right\rVert_{\llr{#2}{#3}}}
\newcommand{\hr}[1]{H_{\rho}^{#1}}
\newcommand{\nhr}[2]{\left\lVert #1 \right\rVert_{\hr{#2}}}
\newcommand{\lhr}[2]{L^{\infty,#1}\hr{#2}}

\newtheorem{thm}{Theorem}[section]
\newtheorem{lem}[thm]{Lemma}
\newtheorem{pp}[thm]{Proposition}
\newtheorem{cor}[thm]{Corollary}
\newtheorem{dfn}[thm]{Definition}

\numberwithin{equation}{section}

\title{A construction of blow up solutions for co-rotational wave maps}
\author{C\u at\u alin I. C\^arstea}
\date{}

\begin{document}
\maketitle

\begin{abstract}
The existence of co-rotational finite time blow up solutions to the wave map problem from $\mathbb{R}^{2+1}\to N$, where $N$ is a surface of revolution with metric $d\rho^2+g(\rho)^2d\theta^2$, $g$ an entire function, is proven. These are of the form $u(t,r)=Q(\lambda(t)t)+\mathcal{R}(t,r)$, where $Q$ is a time independent solution of the co-rotational wave map equation $-u_{tt}+u_{rr}+r^{-1}u_r=r^{-2}g(u)g'(u)$, $\lambda(t)=t^{-1-\nu}$, $\nu>1/2$ is arbitrary, and $\mathcal{R}$ is a term whose local energy goes to zero as $t\to0$.
\end{abstract}

\section{Introduction}

In the following wave maps (see \cite{SS}) from $\R^{2+1}$ into a surface of revolution $N$ (with some restrictions on the metric that will be made explicit below), which are also co-rotational, will be considered. The wave map equation reduces in this case to: 
\begin{equation}
-u_{tt}+u_{rr}+\frac{1}{r}u_r=\frac{1}{r^2}f(u),
\end{equation}
with $r>0$ and where the right hand side is related to the metric of $N$ (see bellow). This equation will be shown to have blow up solutions 
(solutions for which the $||u||_{\dot{H}^{1/2}}$ norm goes to infinity in finite time) with initial data $(u,u_t)$ in $H^{1+\delta}\times H^\delta$, for some $\delta>0$.

The energy 
\begin{equation}
\mathcal{E}(u)=\int_0^\infty\frac{1}{2}\left[(\partial_t u)^2+(\partial_r u)^2+\frac{g(u)^2}{r^2}\right]r\;dr
\end{equation}
is preserved and the problem is energy critical in the sense that the scaling $u\to u(\lambda t,\lambda r)$ leaves $\mathcal{E}(u)$ invariant. If the local energy with respect to the origin is defined to be
\begin{equation}
\mathcal{E}_{loc}(u)=\int_{r<t}\frac{1}{2}\left[(\partial_t u)^2+(\partial_r u)^2+\frac{g(u)^2}{r^2}\right]r\;dr,
\end{equation}
then it is known (see \cite{ST}) that the solution $u$ will blow up at the origin, as $t\to0$, iff 
\begin{equation}
\liminf_{t\downarrow0}\mathcal{E}_{loc}(u)(t)>0.
\end{equation}

In \cite{S2} Struwe has shown that for solutions $u$ with $C^\infty$ data that have blow up at $t_0$ there exist sequences $r_i\downarrow0$ and 
$t_i\uparrow t_0$ such that $r_i/t_i\to0$ and $u_i(t,r)=u(t_i+r_it,r_ix)\to u_\infty(x)$, where $u_\infty$ is a non-constant time indepedent solution of the wave map equation. This motivates the construction detailed in this paper which 
 produces a solution of the wave map equation which inside the light cone $r<t$ is of the form
\begin{equation}
u(t,r)=Q(\lambda(t) r)+u^e(t,r)+\epsilon(t,r),
\end{equation}
where $Q$ is a finite energy, non-trivial stationary solution of the wave map equation (a harmonic map) and $\lambda(t)=t^{-1-\nu}$, $\nu>1/2$. The first term is the one for which 
\begin{equation}
\liminf_{t\downarrow0}\mathcal{E}_{loc}(Q(\lambda(t) r))(t)>0.
\end{equation}
The second term is ``large'', but does not cancel the energy concentration of the first term. The last term is ``small''.

The proof of the main result (Theorem \ref{the-theorem} bellow) follows very closely the work of Krieger, Schlag, and T\u ataru (\cite{KST}) in the particular case when the surface of revolution $N$ is the sphere. Indeed, certain portions of this paper are nearly identical to the ones in \cite{KST}.

The section \ref{ap-sol} here corresponds to section 3 in \cite{KST} and it deals with iteratively constructing corrections to $u_0=Q(\lambda(t)r)$, which will form the $u^e$ term. The procedure is split into four steps which alternate constructions of additive corrections (by two different methods) with estimations of the errors made. Here and in Appendix \ref{appendix-A} is where most of the original contribution of the paper is concentrated. One of the differences from \cite{KST} is in the spaces introduced in subsection \ref{somespaces} below. Though only slightly changed, the definitions given here should also be used to replace the ones in \cite{KST} in order to make some of the computations there meaningful. The computations of the errors corresponding to each of the succesive approximate solutions are also new, as the right hand side term of the wave map equation is more general here. 

Section \ref{pe-eq} corresponds to section 4 in \cite{KST}. In it an equation satisfied by $\epsilon$ is derived. Section \ref{tr-id} (section 6 in \cite{KST}) deals with rewriting this equation as a transport-like equation for the generalized Fourier transform of $\epsilon$ corresponding to a self adjoint operator $\el$, which is a conjugate of the linerization of the spatial part of the wave map equation. The term $\epsilon$ is then obtained in sections 
\ref{fi-eq} and \ref{no-te} (sections 7, 8, and 9 in \cite{KST}) by means of a contraction principle argument. Finally, the proof of Theorem 
\ref{the-theorem} is finalized in section \ref{end}. 

Appendix \ref{appendix-A} corresponds mostly to section 5 in \cite{KST} and it contains an analysis of the spectral theory of the operator $\el$ mentioned above. This is based on results by Gesztesy and Zinchenko (\cite{GZ}) on the spectral theory of Schr\"odinger operators with certain singular potentials. It is due to the fact that the same expansions (see Proposition \ref{phi-expansion-proposition}) can be derived for the generalized Fourier basis of $\el$ as in the particular case of $N=S^2$ that  sections \ref{tr-id}--\ref{end} are essentially identical to their correspondents in \cite{KST}. The original contribution here lies mostly in the proof of Proposition \ref{phi-expansion-proposition}. The Lemma \ref{lel2} is also new as it deals with establishing the properties of a certain convenient system of fundamental solutions for $\el$. This was not necessary in \cite{KST} since there explicit formulas for these solutions are available. The results of Lemma \ref{lel2} are essential for the first step of the iterative procedure of section \ref{ap-sol}.

See also the introduction to \cite{KST} for a discussion of the history of the problem and a more in depth analysis of the motivation for the method. Krieger, Schlag and T\u ataru have also applied the same method to the $H^1(\R^3)$ critical focusing semilinear wave equation in \cite{KST2} and to the 
critical Yang--Mills problem in \cite{KST3}.

I would like thank Prof. C. Kenig and Prof. W. Schlag. Fruitful discussions with both have made this work possible.

\section{Setup}\label{setup}

\subsection{The Manifold}

Let $N$  be a compact surface of revolution, whith Riemannian metric 
\begin{equation}ds^2=d\rho^2+g(\rho)^2d\theta^2.\end{equation}
 If $N$ is produced by rotating the graph of the function 
\begin{equation}y=y(x),\quad y(0)=0,\quad y(x_{M})=0,\end{equation}
 around the $x$-axis, then $\rho$ is the arclength on the graph of the function, $d\rho^2=dx^2+dy^2$. Also,
$y(\rho)=g(\rho)$, hence $dx^2=d\rho\sqrt{1-g'(\rho)^2}$ and $|g'(\rho)|<1$ for any $\rho\in(0,\rmx)$.

In order for the graph of $y=y(x)$ to generate a surface of revolution, it has to be true that ${dy}/{dx}\to\infty$, as $x\to0^+$, and 
${dy}/{dx}\to-\infty$, as $x\to {x_{M}}^-$. Since 
\begin{equation}\frac{dy}{dx}={g'(\rho)}/{\sqrt{1-g'(\rho)^2}},\end{equation}
 it follows that $g'(\rho)\to1$, as $\rho\to0^+$, and
$g'(\rho)\to-1$, as $\rho\to\rmx^-$. It also has to be true that $g$ is an odd function of $\rho$ and of $(\rmx-\rho)$. Therefore it can be extended to a smooth periodic function of period $2\rmx$. 

Throughout this paper, the function $g$ is assumed to have the folowing properties:
\begin{itemize}
\item[i)] $g:\R\to\R$ is entire;
\item[ii)] $g$ is an odd function of $\rho$ and of $\rmx-\rho$;
\item[iii)] $|g'(\rho)|<1$ for all $\rho\in(0,\rmx)$;
\item[iv)] $g'(0)=1$, $g'(\rmx)=-1$.
\end{itemize}
Note then that $g$ can be written as $g(\rho)=\rho G(\rho^2)$, or as $g(\rho)=(\rmx-\rho) \widetilde G\left((\rmx-\rho)^2\right)$, where $G$ and $\widetilde G$ are entire functions, $G(0)=1$, $\widetilde G(0)=-1$.

Let $f(\rho)=g(\rho)g'(\rho)$. This function is also entire, odd, and can be written as $f(\rho)=\rho F(\rho^2)$, or as
$f(\rho)=(\rmx-\rho)\widetilde F\left((\rmx-\rho)^2\right)$, where $F$ and $\widetilde F$ are entire functions, $F(0)=1$, $\widetilde F(0)=1$.

\subsection{The Equation}

Co-rotational wave maps from $\R^{2+1}$ into $N$ are of the form $(t,r,\theta)\to(u(t,r),\theta)$, where $u$ satisfies the following equation:
\begin{equation}\label{wm}
-\partial_t^2 u+\partial_r^2 u+\frac{1}{r}\partial_r u=\frac{f(u)}{r^2}.
\end{equation}
The energy of $u$ is
\begin{equation}
\mathcal{E}(u)=\int_0^\infty\frac{1}{2}\left[(\partial_t u)^2+(\partial_r u)^2+\frac{g(u)^2}{r^2}\right]r\;dr
\end{equation}
and it is constant in time.

\subsection{The Harmonic Map}

Note that for any stationary solution $u$ of \eref{wm}, the following quantity is independent of $r$:
\begin{equation}
\left(r\partial_r u\right)^2-g(u)^2=C.
\end{equation}
If such a solution is to have finite energy, then it is necessary that $C=0$. It follows then that either $r\partial_ru=g(u)$, or $r\partial_ru=-g(u)$.

A stationary solution of the equation (\ref{wm}) is called a harmonic map. As seen just above, harmonic maps with finite energy are solutions of one of two first order ODE and therefore can be specified uniquely by a choice of sign in $r\partial_ru=\pm g(u)$ and the value they take at $r=1$ (for example). Let $Q$ be the solution of:
\begin{equation}\label{hm}
r\partial_r Q=g(Q),\quad Q(1)=1.
\end{equation}
It is clear that $\lim_{r\to0^+}Q(r)=0$, and $\lim_{r\to\infty}Q(r)=\rmx$. With the ansatz $Q(r)=r\mathcal{Q}(r^2)$, equation (\ref{hm}) becomes
\begin{equation}
2r^2\partial_{r^2}\mathcal{Q}(r^2)=\mathcal{Q}(r^2)\left[G(r^2\mathcal{Q}(r^2)^2)-1\right],
\end{equation}
therefore $\mathcal{Q}$ must be an analytic function of $r^2$, 
$\mathcal{Q}(0)>0$. 

Similarly, notice that with the change of variable $l=1/r$, equation (\ref{hm}) can be written as:
\begin{equation}
l\partial_l(\rmx-Q)=(\rmx-Q)\widetilde G((\rmx-Q)^2),
\end{equation}
and, proceeding as above, it follows that $Q(r)=\rmx-(1/r)\widetilde{\mathcal{Q}}(1/r^2)$, where $\widetilde{\mathcal{Q}}$ is analytic, 
$\widetilde{\mathcal{Q}}(0)>0$.

From equations (\ref{hm}) and (\ref{wm}) it follows that 
\begin{equation}
Q''=-\frac{1}{r^2} g(Q)\left(1-g'(Q)\right),
\end{equation}
so $Q'(r)$ is decreasing. Also,
\begin{equation}
\left(r^2 Q'\right)'=g(Q)\left(1+g'(Q)\right),
\end{equation}
so $r^2Q'(r)$ is increasing.

To summarize,
\begin{lem}\label{lemma-Q}
 The chosen harmonic map has the following properties:
\begin{itemize}
\item[i)] $Q'(r)$ is decreasing and $r^2Q'(r)$ is increasing;
\item[ii)] $Q(r)=r\mathcal{Q}(r^2)$, with $\mathcal{Q}$ a real-analytic function,  $\mathcal{Q}(0)>0$;
\item[iii)]  $Q(r)=\rmx-(1/r)\widetilde{\mathcal{Q}}(1/r^2)$, with $\widetilde{\mathcal{Q}}$ a real-analytic function, 
$\widetilde{\mathcal{Q}}(0)>0$.
\end{itemize}
\end{lem}

\subsection{The Theorem}

Define the local energy of a solution $u$ of \eref{wm} with respect to the origin and at time $t$ to be:
\begin{equation}
\mathcal{E}_{loc}(u)(t)=\int_{r<t}\left[\frac{1}{2}(u_t^2+u_r^2)+\frac{g(u)^2}{2r^2}\right]r\;dr
\end{equation}

\begin{thm}\label{the-theorem}
Let $\nu>1/2$ be arbitrary and $t_0>0$ be suficiently small. Define $\lambda(t)=t^{-1-\nu}$ and fix a large integer $N$. Then there exists a function
$u^e$ satisfying
\begin{equation}
u^e\in C^{\nu+1/2-}(\{t_0>t>0,|x|<t\}),
\end{equation}
\begin{equation}
\quad\mathcal{E}_{loc}(u^e)(t)\lesssim(t\lambda(t))^{-2}|\log t|^2\text{ as }t\to0,
\end{equation}
and a solution $u$ of \eref{wm} in $[0,t_0]$ which is of the form
\begin{equation}
u(t,r)=Q(\lambda(t)r)+u^{e}(t,r)+\epsilon(t,r),\quad 0\leq r\leq t,
\end{equation}
where $\epsilon$ decays at $t=0$. More precisely,
\begin{equation}
\epsilon\in t^N H_{loc}^{1+\nu-}(\R^2),\quad\epsilon_t\in t^{N-1}H_{loc}^{\nu-}(\R^2),\quad
\mathcal{E}_{loc}(\epsilon)(t)\lesssim t^N\text{ as }t\to0,
\end{equation}
with spatial norms that are uniformy controlled as $t\to0$. Also, $u(0,t)=0$ for all $0<t<t_0$. The solution $u(t,r)$ extends as an $H^{1+\nu-}$ solution to all $\R^2$.
\end{thm}

\section{Approximate Solutions}\label{ap-sol}

Let $\lambda(t)=t^{-1-\nu}$, $\nu>1/2$, $R=\lambda(t)r$. In this section a sequence $u_k$ of approximate solutions of \eref{hm} will be constructed. 
For each of these the corresponding error is defined to be: 
\begin{equation}
e_k=\op u_k-\frac{1}{r^2}f(u_k).
\end{equation}
The first element of the sequence is $u_0(t,r)=Q(R)$. For a large enough $N$, $u_N-u_0$ will be the $u^e$ of Theorem \ref{the-theorem}.

To motivate the particular construction, suppose that the sought solution of \eqref{wm} is of the form:
\begin{equation}
u=u_k+\epsilon,
\end{equation}
with $\epsilon$ small. Then
\begin{equation}
\left(-\partial_t^2+\partial_r^2+\frac{1}{r}\partial_r \right)\epsilon-\frac{1}{r^2}f'(u_k)\epsilon\approx-e_k.
\end{equation}
Two different approximations of this linearized equation will be used. The first assumes  the time derivative to be unimportant and also approximates $u_k\approx u_0$, replacing $f'(u_k)$ by $f'(u_0)$. The second one retains the time derivative, but assumes that $u_k\approx u_0(\infty)=\rmx$, replacing 
$f'(u_k)$ by $1$, as would be the case if $r\approx t$ and $t$ would be close to zero. Succesive corrections $v_k=u_k-u_{k-1}$ to the approximate solutions will be constructed using these two ideas alternatively, that is the $v_k$'s will be required to solve
\begin{equation}
\left(\partial_r^2+\frac{1}{r}\partial_r -\frac{1}{r^2}f'(u_0)\right)v_{2k+1}=-e_{2k}^0
\end{equation}
and
\begin{equation}\label{equation-v-2k}
\left(-\partial_t^2+\partial_r^2+\frac{1}{r}\partial_r -\frac{1}{r^2}\right)v_{2k+2} =-e_{2k+1}^0,
\end{equation}
with zero Cauchy data at $r=0$ and
where $e_k^0$ is the ``principal part'' of $e_k$, in a sense that will be detailed bellow. 

The conclusion of this section requires the introduction of certain spaces of functions on the light cone. It can be found stated in equations 
\eref{induction-1}--\eref{induction-4}.

This section mirrors section 3 of \cite{KST}. Step 3, in particular is virtually identical to the reference as \eref{equation-v-2k} does not depend on the particular geometry of the surface of revolution. The main difference lies in error estimates of Steps 2 and 4. It is in the course of these two steps that the assumption that $g$ is entire is necessary. 

Note that in the definitions of the spaces of functions in the following subsection three ``$b$'' parameters are used ($b$, $b_1$, $b_2$), instead of one as in \cite{KST}. The definitions given in \cite{KST} should be replaced by the ones bellow. 
Certain other typos have been fixed here.

\subsection{Some Spaces}\label{somespaces}

Before proceeding with the construction, a few spaces of functions need to be introduced. Let
\begin{equation}
\mathcal{C}_0=\{(t,r):0\leq r\leq t, 0<t<t_0\}
\end{equation}
be  a truncated forward light cone on which the $u_k$'s will be defined.

\begin{dfn}\label{ds1}
For $i \in \mathbb{N}$ let $j(i)=i$ if $\nu$ is irrational, and $j(i)=2i^2$ if $\nu$ is rational.
 $\Q$ is the algebra of continuous functions $q:[0,1]\to\R$ with the following properties:
\begin{itemize}
\item[i)] $q$ is analytic in $[0,1)$ with even expansion at 0;
\item[ii)] near $a=1$ there is an absolutely convergent expansion of the form:
\begin{multline}\label{defq}
q=q_0(a)\\
+\sum_{i=1}^\infty\left( (1-a)^{(2i-1)\nu+\frac{1}{2}}\sum_{j=0}^{j(2i-1)}q_{2i-1,j}(a)(\log(1-a))^j\right.\\\left.
+(1-a)^{2i\nu+1}\sum_{j=0}^{j(2i)}q_{2i,j}(a)(\log(1-a))^j \right)
\end{multline}
with analytic coefficients $q_0$, $q_{i,j}$.
\end{itemize}
\end{dfn}

\begin{dfn}\label{ds2}
With $j(i)$ as above, $\Q'$ is the space of functions $q:[0,1]\to\R$ with the following properties:
\begin{itemize}
\item[i)] $q$ is analytic in $[0,1)$ with even expansion at 0;
\item[ii)] near $a=1$ there is an absolutely convergent expansion of the form:
\begin{multline}\label{defqp}
q=q_0(a)\\
+\sum_{i=1}^\infty\left( (1-a)^{(2i-1)\nu-\frac{1}{2}}\sum_{j=0}^{j(2i-1)}q_{2i-1,j}(a)(\log(1-a))^j\right.\\\left.
+(1-a)^{2i\nu+1}\sum_{j=0}^{j(2i)}q_{2i,j}(a)(\log(1-a))^j \right)
\end{multline}
with analytic coefficients $q_0$, $q_{i,j}$.
\end{itemize}
\end{dfn}

\begin{dfn}\label{ds3}
\begin{itemize}
\item[i)] $\Q_m$ is the sub-algebra of $\Q$ defined by the requirement that $q_{ij}(1)=0$ if $i\geq 2m+1$ and $i$ is odd;
\item[ii)] $\Q_m'$ is the sub-space of $\Q'$ defined by the requirement that $q_{ij}(1)=0$ if $i\geq 2m+1$ and $i$ is odd.
\end{itemize}
\end{dfn}

\begin{lem}\label{ls1}
\begin{itemize}
\item[i)] $\Q\subset\Q'$ and $\Q_m\subset\Q_m'$;
\item[ii)] $\Q_m\subset\Q_{m+1}$, $\Q_m'\subset\Q_{m+1}'$.
\end{itemize}
\end{lem}
\begin{proof}
Note that the only difference between the definitions of the $\Q$ spaces and the $\Q'$ spaces is that a power $(1-a)^{1/2}$ appears in the fist term inside the bracket in \eref{defq}, while in the same place in \eref{defqp} there is a power $(1-a)^{-1/2}$. i) follows from:
\begin{equation}
(1-a)^{1/2}=(1-a)(1-a)^{-1/2}.
\end{equation}
ii) is obvious from Definition \ref{ds3}.
\end{proof}

\begin{dfn}\label{ds4}
$\s{m}{k}{l}$ is the class of functions $v:[0,\infty)\to\R$ with the following properties:
\begin{itemize}
\item[i)] $v$ vanishes of order $m$ at $R=0$, and $v(R)=R^m\sum_{j=0}^\infty c_jR^{2j}$ for small $R$;
\item[ii)] $v$ has a convergent expansion near $R=\infty$ of the form:
\begin{equation}
v(R)=\sum_{0\leq j\leq l+i}c_{ij}R^{k-2i}(\log R)^j.
\end{equation}
\end{itemize}
\end{dfn}

Let $B$, $B_1$, $B_2$ be positive constants to be specified shortly.

\begin{dfn}\label{ds5}
$\sq{m}{k}{l}{n}$ is the class of analytic functions $v:[0,\infty)\times[0,1]\times[0,B]\times[0,B_1]\times[0,B_2]\to\R$ with the following properties:
\begin{itemize}
\item[i)] $v$ is analytic as a function of $R$, $b$, $b_1$, $b_2$
\begin{equation}
v:[0,\infty)\times[0,B]\times[0,B_1]\times[0,B_2]\to\Q_n;
\end{equation}
\item[ii)] $v$ vanishes of order $m$ at $R=0$ and has a convergent expansion
\begin{equation}
v(R,a,b,b_1,b_2)=R^m\sum_{j=0}^\infty c_j(a,b,b_1,b_2)R^{2j};
\end{equation}
\item[iii)] $v$ has a convergent expansion near $R=\infty$ of the form
\begin{equation}
v(R,\cdot,b,b_1,b_2)=\sum_{0\leq j\leq l+i}c_{ij}(\cdot,b,b_1,b_2)R^{k-2i}(\log R)^j.
\end{equation}
where the coefficients $c_{ij}:[0,B]\times[0,B_1]\times[0,B_2]\to\Q_n$ are analytic with respect to $b$, $b_1$, $b_2$.
\end{itemize}
$\sqp{m}{k}{l}{n}$ is defined similarly.
\end{dfn}

Here is a list of elementary, but useful, properties of these spaces:

\begin{lem}\label{ls2}
\begin{itemize}
\item[i)] $\sq{m+2}{k}{l}{n}\subset\sq{m}{k}{l}{n}$;
\item[ii)] $\sq{m}{k}{l}{n}\subset\sq{m}{k}{l+1}{n}$;
\item[iii)] $\sq{m}{k}{l}{n}\subset\sq{m}{k+2}{l-1}{n}$;
\item[iv)] $\sq{m}{k}{l}{n}\subset\sqp{m}{k}{l}{n}$.
\end{itemize}
All but the last one are also properties of $\sqp{k}{l}{m}{n}$.
\end{lem}

With the notations $R=\lambda(t)r$, $a=r/t=(t\lambda)^{-1}R$, $b=(t\lambda)^{-2}[\log(2+R^2)]^2$, $b_1=(t\lambda)^{-2}[\log(2+R^2)]$,
$b_2=(t\lambda)^{-2}$, if $(t,r)\in\mathcal{C}_0$, then there are positive constants $B$, $B_1$, $B_2$ such that $b\in[0,B]$,
$b_1\in[0,B_1]$, and $b_2\in[0,B_2]$.

\begin{dfn}\label{ds6}
$\isq{m}{k}{l}{n}$ is the class of analytic functions  $w$ defined on the cone $\mathcal{C}_0$ which can be represented as
\begin{equation}
w(t,r)=v(R,a,b,b_1,b_2),\qquad v\in\sq{m}{k}{l}{n}.
\end{equation}
The definition of $\isqp{m}{k}{l}{n}$ is similar.
\end{dfn}

Note that the representations in the above definition are not at all unique.

\subsection{Two Useful Lemmas}

The following results will be useful throughout this section.

\begin{lem}\label{useful-lemma-1}
$f^{(2k)}(Q(R))\in IS^1(R^{-1})$ and $f^{(2k+1)}(Q(R))\in IS^0(1)$.
\end{lem}

\begin{proof}
$f^{(2k)}(\rho)$ has an odd expansion in $\rho$ and also in $(\rmx-\rho)$. Plugging in $Q$ the first half of the result follows from Lemma \ref{lemma-Q}. The case of $f^{(2k+1)}$ is similar, but with even expansions.
\end{proof}

\begin{lem}\label{useful-lemma-2}
If 
\begin{equation}
z\in\tl{2}\isq{1}{}{}{},
\end{equation}
then
\begin{equation}\label{q-plus-z-2k}
f^{(2k)}(Q(R)+z(R))\in \tl{2}\isq{1}{}{}{}
\end{equation} and 
\begin{equation}\label{q-plus-z-2k+1}
f^{(2k+1)}(Q(R)+z(R))\in IS^0(1,\Q).
\end{equation}
\end{lem}

\begin{proof}
First expand
\begin{equation}\label{f-2k-expand-a}
f^{(2k)}(Q+z)=\sum_{l\geq0}\frac{1}{l!}f^{2k+l}(Q)z^l.
\end{equation}
Note that 
\begin{multline}
z^2\in\tl{4}\isq{2}{2}{2}{}\\
\subset\tl{2}b\isq{2}{2}{0}{}+\tl{2}b_1\isq{2}{2}{0}{}\\+\tl{2}b_2\isq{2}{2}{0}{}\\
\subset\tl{2}\isq{2}{2}{0}{}\subset a^2 IS^{0}(1,\Q).
\end{multline}
Now
\begin{equation}
f^{(2k+2m+1)}(Q)z\in\tl{2}\isq{1}{}{}{}
\end{equation}
and
\begin{multline}\label{f-2k-expand-b}
f^{(2k+2m)}(Q)z^2\in\tl{4}\isq{3}{}{2}{}\subset\tl{2}\isq{1}{}{}{}.
\end{multline}
Combining equations \eref{f-2k-expand-a}--\eref{f-2k-expand-b} yields \eref{q-plus-z-2k}.

To prove \eref{q-plus-z-2k+1} proceed similarly by expanding
\begin{equation}
f^{(2k+1)}(Q+z)=\sum_{l\geq0}\frac{1}{l!}f^{2k+l+1}(Q)z^l.
\end{equation}
Similar computations to the ones above give the result.

\end{proof}

\subsection{Step 0}

As is mentioned above, the first element of the squence of approximate solutions is $u_0=Q(R)$. The corresponding error is then:
\begin{align}
e_0&=\op u_0-\frac{1}{r^2}f(u_0)\nonumber\\
&=-\partial_t^2Q(\lambda(t)r)\nonumber\\
&=-\partial_t[r\lambda'(t)Q'(\lambda(t)r)]\nonumber\\
&=-r\lambda''(t)Q'(R)-r^2\lambda'(t)^2Q''(R)\nonumber\\
&=-\frac{1}{t^2}\left[(1+\nu)(2+\nu)RQ'(R)+(1+\nu)^2R^2Q''(R)\right].
\end{align}
Therefore, $t^2e_0\in IS^1\left(R^{-1}\right)$.

\subsection{Induction}

The approximate solutions will be constructed by adding succesive corrections to $u_0$. With the notation $v_k=u_k-u_{k-1}$, it will be inductively shown that 
\begin{align}
v_{2k-1}&\in\frac{1}{(t\lambda)^{2k}}\isq{3}{}{2k-1}{k-1}\label{induction-1}\\
t^2e_{2k-1}&\in\frac{1}{(t\lambda)^{2k}}\isqp{1}{}{2k-1}{k-1}\label{induction-2}\\
v_{2k}&\in\frac{1}{(t\lambda)^{2k+2}}\isq{3}{3}{2k-1}{k}\label{induction-3}\\
t^2e_{2k}&\in\frac{1}{(t\lambda)^{2k}}\left[\isq{1}{-1}{2k}{k}+\right.\nonumber\\
&+b\isqp{1}{}{2k-1}{k}+\nonumber\\
&+b_1\isqp{1}{}{2k-1}{k}+\nonumber\\
&+\left.b_2\isqp{1}{}{2k-1}{k}\right]\label{induction-4}
\end{align}
The exact method for constructing the $v_k$'s will be described bellow. In the following, for a fixed $k$, it will be assumed that the above hold for $k$ and for any smaller natural number.

\subsection{Step 1}

It is assumed that
\begin{multline}
t^2e_{2k-2}\in\frac{1}{(t\lambda)^{2k-2}}\Bigg[\isq{1}{-1}{2k-2}{k-1}+\\+
\sum_{\beta=b,b_1,b_2}\beta\,\isqp{1}{}{2k-3}{k-1}\Bigg]
\end{multline}
Choose the ``principal part''   $e_{2k-2}^0$ by setting $b=b_1=b_2=0$ in a representation of $e_{2k-2}$ (see Definition \ref{ds6}). Then
\begin{equation}
t^2e_{2k-2}^0\in\tl{2k-2}\isq{1}{-1}{2k-2}{k-1},
\end{equation}
and
\begin{multline}
t^2e_{2k-2}^1=t^2\left(e_{2k-2}-e_{2k-2}^0\right)\in\\\in
\sum_{\beta=b,b_1,b_2}\beta\,\frac{1}{(t\lambda)^{2k-2}}\left[\isq{1}{-1}{2k-2}{k-1}+\right.\\\left.+
\isqp{1}{}{2k-3}{k-1}\right].
\end{multline}
Replacing the $b$, $b_1$, $b_2$ by their definitions, it follows that
\begin{multline}
t^2e_{2k-2}^1\in\tl{2k}\left[\isqp{1}{-1}{2k}{k-1}+\right.\\\left.+\isqp{1}{}{2k-1}{k-1}\right],
\end{multline}
so
\begin{equation}\label{eone-step2}
t^2e_{2k-2}^1\in\tl{2k}\isqp{1}{}{2k-1}{k-1}.
\end{equation}
This will be useful later.

Let
\begin{equation}
L=\opR-\frac{f'(u_0)}{R^2}.
\end{equation}
Keeping $a$, $b$, $b_1$, $b_2$ fixed, define $v_{2k-1}$ to be the solution of
\begin{equation}
(t\lambda)^2Lv_{2k-1}=-t^2e_{2k-2}^0,
\end{equation}
with vanishing Cauchy data at $R=0$.

\begin{lem}
The solution of $Lv=\varphi\in\s{1}{-1}{2k-2}$, with $v(0)=v'(0)=0$, has the regularity
\begin{equation}
v\in\s{3}{}{2k-1}.
\end{equation}
\end{lem}

\begin{proof}
\textbf{Behavior at $R\sim0$.} Close to zero, 
\begin{equation}
\varphi(R)=\sum_{k=0}^\infty \varphi_k R^{2k+1},\quad f'(u_0(R))=1+\sum_{k=1}^\infty f_k R^{2k}
\end{equation}
Make the ansatz 
\begin{equation}\label{ansatzV}
v(R)=\sum_{k=1}^\infty V_kR^{2k+1}.
\end{equation}
Since $\partial_R^2+R^{-1}\partial_R=R^{-1}\partial_R R\partial_R$, it follows that the $V_k$ need to satisfy
\begin{equation}
((2k+1)^2-1)V_k=\phi_{k-1}+\sum_{l=1}^{k-1}f_lV_{k-l},\quad\forall k\geq1.
\end{equation}
This system can be solved to find $V_k$ such that the sum in \eref{ansatzV} converges absolutely in a neighborhood of zero. Such a  $v$ will vanish of order 3 at zero.

\textbf{Behavior at $R\sim\infty$.} Notice that $v$ has to satisfy $\el\sqrt{R}v=-\sqrt{R}\varphi$, where
\begin{equation}
\el=-\partial_r^2+\frac{3}{4r^2}+V(r);\quad\quad V(r)=-\frac{1}{r^2}\left[1-f'(Q(r))\right].
\end{equation}
Using the fundamental system of $\el$ from Lemma \ref{lel2}, $v$ can be written as:
\begin{multline}
v(R)=c_\phi\frac{1}{\sqrt{R}}\phi_0(R)+c_\theta\frac{1}{\sqrt{R}}\theta_0(R)+\\+
\frac{1}{2}\frac{1}{\sqrt{R}}\theta_0(R)\int_1^R\phi_0(S)\sqrt{S}\varphi(S)\;dS-\\-
\frac{1}{2}\frac{1}{\sqrt{R}}\phi_0(R)\int_1^R\theta_0(S)\sqrt{S}\varphi(S)\;dS
\end{multline}
By Lemma \ref{lel2}, 
\begin{align}
R^{-1/2}\phi_0(R)&\in S(R^{-1}),\\ R^{-1/2}\theta_0(R)&\in S(R^{-1}),\\ R^{1/2}\phi_0(R)\varphi(R)&\in\s{}{-1}{2k-2},\\
R^{1/2}\theta_0(R)\varphi(R)&\in\s{}{}{2k-2}.
\end{align}
Therefore
\begin{align}
\int_1^R\phi_0(S)\sqrt{S}\varphi(S)\;dS&\in S\left((\log R)^{2k-1}\right)\\
\int_1^R\theta_0(S)\sqrt{S}\varphi(S)\;dS&\in\s{}{2}{2k-2}+S\left((\log R)^{2k-1}\right)\nonumber\\
&\subset \s{}{2}{2k-2}.
\end{align}
Since $v$ is sought such that it has zero Cauchy data, then $c_\phi=c_\theta=0$.
Putting all these together, it follows that
\begin{equation}
v\in\s{}{}{2k-1}.
\end{equation}
\end{proof}

An immediate consequence of the previous Lemma is that
\begin{equation}
v_{2k-1}\in\tl{2k}\isq{3}{}{2k-1}{k-1}.
\end{equation}

\subsection{Step 2}

The error corresponding to $v_{2k-1}$ is:
\begin{equation}
e_{2k-1}=e_{2k-2}^1+N_{2k-1}(v_{2k-1})+E^tv_{2k-1}+E^av_{2k-1},
\end{equation}
where
\begin{equation}
N_{2k-1}(v)=\frac{1}{r^2}\left[f'(u_0)v-f(u_{2k-2}+v_{2k-1})-f(u_{2k-1})\right],
\end{equation}
$E^tv_{2k-1}$ designates the terms in $\partial_t^2v_{2k-1}$ with no derivatives on the $a$ variable, and $E^av_{2k-1}$ designates the terms in 
$\op v_{2k-1}$ with at least one derivative on the $a$ variable.

\subsubsection{The $N_{2k-1}(v_{2k-1})$ term}

First write
\begin{multline}\label{Ntermsplit}
t^2N_{2k-1}(v_{2k-1})=-a^{-2}\left[\left(f(u_{2k-2}+v_{2k-1})-f(u_{2k-2})-\right.\cut\left.-f'(u_{2k-2})v_{2k-1}\right)+
\left(f'(u_{2k-2})-f'(u_0)\right)v_{2k-1}\right]=\\=-a^{-2}\left[I+II\right].
\end{multline}
For $l\leq k$,
\begin{align}
v_{2l-1}&\in\tl{2l}\isq{3}{}{2l-1}{l-1}\nonumber\\
&\subset\tl{2}b^{l-1}\isq{3}{}{}{l-1}\nonumber\\
&+\tl{2}b^{l-2}b_2\isq{3}{}{}{l-1}\nonumber\\
&+\cdots\cdots\cdots\cdots\cdots\cdots\cdots\nonumber\\
&+\tl{2}b_2^{l-1}\isq{3}{}{}{l-1}\nonumber\\
&\subset\tl{2}\isq{3}{}{}{l-1},\label{eq-for-lemma-f-1}
\end{align}
and for $l<k$,
\begin{align}
v_{2l}&\in\tl{2l+2}\isq{3}{3}{2l-1}{l}\nonumber\\
&\subset a^2\tl{2l}\isq{1}{}{2l-1}{l}\nonumber\\
&\subset\tl{2}\isq{1}{}{}{l}.\label{eq-for-lemma-f-2}
\end{align}
Therefore, for any $l<k$, 
\begin{equation}\label{eq-for-lemma-f-3}
(u_{2l+1}-u_0),\, (u_{2l}-u_0)\in\tl{2}\isq{1}{}{}{k-1}.
\end{equation}

Returning to \eref{Ntermsplit}, 
\begin{equation}
I=v_{2k-1}^2\sum_{l\geq2}\frac{1}{l!}f^{(l)}(u_0+(u_{2k-2}-u_0))v_{2k-1}^{l-2}.
\end{equation}
Note that, since $v_{2k-1}\in\tl{2}\isq{1}{}{}{}$, $v_{2k-1}^2\in IS^0(1,\Q)$.
For the even terms in the expansion above, using Lemma \ref{useful-lemma-2},
\begin{equation}
f^{(2m)}(u_{2k-2})(z^2)^{m-1}\in\tl{2}\isq{1}{}{}{}
\end{equation}
and for the odd ones
\begin{equation}
f^{(2m+1)}(u_{2k-2})z(z^2)^{m-1}\in\tl{2}\isq{1}{}{}{}.
\end{equation}
Therefore
\begin{multline}
I\in\tl{4k+2}\isq{7}{3}{4k-1}{k-1}\\
\subset\tl{2k+2}\isq{7}{3}{2k-1}{k-1}\\
\subset a^2\tl{2k}\isqp{5}{}{2k-1}{k-1}.
\end{multline}
So
\begin{equation}\label{NtermI}
a^{-2}I\in\tl{2k}\isqp{5}{}{2k-1}{k-1}.
\end{equation}

Now
\begin{equation}
II=v_{2k-1}\sum_{l\geq2}\frac{1}{(l-1)!}f^{(l)}(u_0)(u_{2k-2}-u_0)^{l-1}.
\end{equation}
From computations above it follows that
\begin{equation}
(u_{2k-2}-u_0)^2\in a^2IS^0(1,\Q).
\end{equation}
Then, using Lemma \ref{useful-lemma-1},
\begin{multline}
f^{(2m)}(u_0)(u_{2k-2}-u_0)^{2m-1}\in\tl{2}IS^2(\log R,\Q)\\\subset\tl{2}IS^2(R^2)\subset a^2 IS^0(1,\Q)
\end{multline}
and
\begin{equation}
f^{(2m+1)}(u_0)(u_{2k-2}-u_0)^{2m}\in a^2 IS^0(1,\Q).
\end{equation}
Therefore
\begin{equation}\label{NtermII}
a^{-2}II\subset\tl{2k}\isqp{3}{}{2k-1}{k-1}.
\end{equation}

From \eref{NtermI} and \eref{NtermII} it follows that
\begin{equation}\label{nv}
t^2N_{2k-1}(v_{2k-1})\in\tl{2k}\isqp{3}{}{2k-1}{k-1}.
\end{equation}

\subsubsection{The $E^tv_{2k-1}$ term}

Recall that $E^tv_{2k-1}=\partial_t^2v_{2k-1}$ with $a$ fixed. Note that there is no dependence on $b$, $b_1$, $b_2$ in $v_{2k-1}$ since $e_{2k-2}^0$ was obtained by setting these to zero. $v_{2k-1}$ can be written as
\begin{equation}
v_{2k-1}=\tl{2k}w(R,a)
\end{equation}
with $w\in\sq{3}{}{2k-1}{k-1}$.
\begin{equation}
\partial_t v_{2k-1}=2k\nu\frac{1}{t}\tl{2k}w+\tl{2k}R\partial_R w\left(-\frac{1+\nu}{t}\right)+(\cdots),
\end{equation}
where the terms left out are those that involve $\partial_a$.
\begin{multline}
\partial_t^2 v_{2k-1}=(2k\nu)^2\frac{1}{t^2}\tl{2k}w-2k\nu\frac{1}{t^2}\tl{2k}w+\\+
2(2k\nu)\frac{1}{t}R\tl{2k}\partial_R w\left(-\frac{1+\nu}{t}\right)+
\frac{1+\nu}{t^2}R\tl{2k}\partial_Rw+\\+\tl{2k}R\partial_R R\partial_R w\left(-\frac{1+\nu}{t}\right)^2+(\cdots).
\end{multline}
Since
\begin{equation}
w,R\partial_Rw,(R\partial_R)^2w\in\sq{3}{}{2k-1}{k-1},
\end{equation}
it follows that
\begin{equation}\label{etv}
t^2E^tv_{2k-1}\in\tl{2k}\isqp{3}{}{2k-1}{k-1}.
\end{equation}

\subsubsection{The $E^av_{2k-1}$ term}

Using the same notation as above, remembering that there is no dependence on $b$, $b_1$, $b_2$ in $v_{2k-1}$, and omitting to write explicitly the terms that will not become part of $E^av_{2k-1}$,
\begin{equation}
\frac{1}{r}\partial_r v_{2k-1}=\tl{2k}w_a\frac{1}{t^2}a^{-1}+(\cdots),
\end{equation}
\begin{multline}
\partial_r^2v_{2k-1}=\tl{2k}\partial_r\left(\frac{1}{t}w_a+\lambda w_R\right)\\
=\tl{2k}\left(\frac{1}{t^2}w_{aa}+\frac{2\lambda}{t}w_{aR}\right)+(\cdots),
\end{multline}
\begin{multline}
\partial_t^2v_{2k-1}=\partial_t\left( 2k\nu\tl{2k}\frac{1}{t}w+\tl{2k}w_a\left(-\frac{a}{t}\right)+\tl{2k}Rw_R\left(-\frac{1+\nu}{t}\right)\right)\\
=\left[ 2k\nu\tl{2k}\frac{2}{t}w_a\left(-\frac{a}{t}\right) +\tl{2k}w_a\left(\frac{2}{t^2}a\right)+\tl{2k}w_{aa}\left(-\frac{a}{t}\right)^2\cut+
\tl{2k}2Rw_{aR}\left(-\frac{a}{t}\right)\left(-\frac{1+\nu}{t}\right)  \right] +(\cdots).
\end{multline}
Putting these together
\begin{multline}
t^2E^av_{2k-1}=\tl{2k}\left[(1-a^2)w_{aa} +[2a(2k\nu-1)+a^{-1}]w_a\cut-2R[(1+\nu)a-a^{-1}]w_{aR}\right].
\end{multline}
Since
\begin{equation}
a\partial_a,a^{-1}\partial_a,(1-a^2)\partial_a^2:\Q_{k-1}\to\Q_{k-1}',
\end{equation}
it follows that
\begin{equation}\label{eav}
t^2E^av_{2k-1}\in\tl{2k}\isqp{3}{}{2k-1}{k-1}.
\end{equation}

To conclude, the results \eref{eone-step2}, \eref{nv}, \eref{etv}, and \eref{eav} imply that
\begin{equation}
t^2e_{2k-1}\in\tl{2k}\isqp{1}{}{2k-1}{k-1}.
\end{equation}

\subsection{Step 3}

Let
\begin{equation}
t^2f_{2k-1}=\frac{R}{(t\lambda)^{2k}}\sum_{j=0}^{2k-1}q_j(a)(\log R)^j=\tl{2k-1}\sum_{j=0}^{2k-1}aq_j(a)(\log R)^j,
\end{equation}
be the sum of the leading terms of the expansion of $e_{2k-1}$ at $R=\infty$, with $b=b_1=b_2=0$. By definition, $q_j\in\Q_{k-1}'$ for al $j$.
Define $w_{2k}$ to be a solution of the equation
\begin{equation}\label{eqw}
t^2\left(-\partial_t^2+\partial_r^2+\frac{1}{r}\partial_r-\frac{1}{r^2}\right)w_{2k}=-t^2f_{2k-1}
\end{equation}
Making the ansatz
\begin{equation}
w_{2k}=\tl{2k-1}\sum_{j=0}^{2k-1}W_{2k}^j(a)(\log R)^j,
\end{equation}
plugging into \eref{eqw}, and matching the corresponding powers of $\log R$, it follows that the $W_{2k}^j$ have to satisfy the equations
\begin{multline}\label{systemW}
t^2\left(-\partial_t^2+\partial_r^2+\frac{1}{r}\partial_r-\frac{1}{r^2}\right)\left(\tl{2k-1}W_{2k}^j(a)\right)\\=
-\tl{2k-1}(aq_j(a)+F_j(a)),
\end{multline}
where, with the convention that $W_{2k}^j=0$ when $j\geq2k$,
\begin{multline}
F_j(a)=(j+1)\left[
(1+\nu)(2\nu(2k-1)-1)W_{2k}^{j+1}+
2(a^{-1}-(1+\nu)a)\partial_a W_{2k}^{j+1}
\right]\\+
(j+1)\left[(j+2)a^{-2}-j(1+\nu)^2\right]w_{2k}^{j+2}.
\end{multline}
Conjugating by $(t\lambda)^{-(2k-1)}$, the  system of equations \eref{systemW} becomes
\begin{multline}\label{systemW2}
t^2\left(-\left(\partial_t+\frac{(2k-1)\nu}{t}\right)^2+\partial_r^2+\frac{1}{r}\partial_r-\frac{1}{r^2}\right)W_{2k}^j(a)\\=
-aq_j(a)-F_j(a).
\end{multline}
With the notation
\begin{equation}
L_\beta=(1-a^2)\partial_a^2+(a^{-1}+2a\beta-2a)\partial_a+(-\beta^2+\beta-a^{-2}),
\end{equation}
writing \eref{systemW2} in terms of derivatives in $a$ yields:
\begin{equation}
L_{(2k-1)\nu}W_{2k}^j=-\left(aq_j(a)+F_j(a)\right).
\end{equation}
Adding the requirement that the Cauchy data at $a=0$ for this system is zero, the solutions will satisfy
\begin{equation}\label{wa3q}
W_{2k}^j\in a^3\Q_k,\qquad j=\overline{0,2k-1}.
\end{equation}
See \cite{KST} for a proof of this fact.

The $w_{2k}$ constructed so far cannot be used as $v_{2k}$ as it is singular at zero. Instead, define
\begin{align}
v_{2k}&=\tl{2k-1}\sum_{j=0}^{2k-1}W_{2k}^j(a)\left(\frac{1}{2}\log(1+R^2)\right)^j\nonumber\\
&=\tl{2k+2}\sum_{j=0}^{2k-1}a^{-3}W_{2k}^j(a)R^3\left(\frac{1}{2}\log(1+R^2)\right)^j.
\end{align}
Then clearly
\begin{equation}
v_{2k}\in\tl{2k+2}\isq{3}{3}{2k-1}{k}.
\end{equation}

\subsection{Step 4}

Define
\begin{align}
t^2e_{2k-1}^0&=\frac{R}{(t\lambda)^{2k}}\sum_{j=0}^{2k-1}q_j(a)\left(\frac{1}{2}\log(1+R^2)\right)^j\nonumber\\
&\in\tl{2k}\isqp{1}{}{2k-1}{k-1}.
\end{align}
The error corresponding to $v_{2k}$ is
\begin{multline}\label{def-e2k}
t^2e_{2k}=t^2(e_{2k-1}-e_{2k-1}^0)\\+t^2\left(e_{2k-1}^0+ \left( -\partial_t^2+\opr-\frac{1}{r^2}\right)v_{2k}\right)+t^2N_{2k}(v_{2k}),
\end{multline}
where
\begin{equation}
N_{2k}(v)=\frac{v}{r^2}-\frac{1}{r^2}\left[f(u_{2k-1}+v)-f(u_{2k-1})\right].
\end{equation}

\subsubsection{The first term of \eref{def-e2k}}

Both $t^2e_{2k-1}$ and $t^2e_{2k-1}^0$ have the same leading order in their expansions at $R=\infty$. Therefore
\begin{equation}
t^2(e_{2k-1}-e_{2k-1}^0)\in\tl{2k}\isqp{1}{-1}{2k}{k-1}.
\end{equation}

Suppose 
\begin{equation}\label{w1}
w\in\isqp{1}{-1}{2k}{k-1}.
\end{equation}
This can be written as 
\begin{equation}
w=(1-a^2)w+\frac{R^2}{(t\lambda)^2}w.
\end{equation}
The first term satisfies
\begin{equation}
(1-a^2)w\in\isq{1}{-1}{2k}{k-1}.
\end{equation}
In the case of the second term
\begin{align}\label{w2}
\frac{R^2}{(t\lambda)^2}w&\in\tl{2}\isqp{3}{}{2k}{k-1}\nonumber\\
&\subset b_1\isqp{3}{}{2k-1}{k-1}\nonumber\\
&+b_2IS^{3}\left(R,\Q_{k-1}'\right).
\end{align}
Applying this to $t^2(e_{2k-1}-e_{2k-1}^0)$, it follows that
\begin{multline}\label{step4-1}
t^2(e_{2k-1}-e_{2k-1}^0)\in\frac{1}{(t\lambda)^{2k}}\Bigg[\isq{1}{-1}{2k}{k-1}\\+
\sum_{\beta=b,b_1,b_2}\beta\,\isqp{1}{}{2k-1}{k-1}\Bigg].
\end{multline}

\subsubsection{The second term of \eref{def-e2k}}

The reason this term is not zero is the replacement of $\log R$ by $\frac{1}{2}\log(1+R^2)$ made above. The second term of \eref{def-e2k}
consists of a sum of expressions of the type
\begin{multline}
\tl{2k-1}\sum_{j=0}^{2k-1}\left\{a^{-2}W_{2k}^j\left[IS^0(R^{-2})(\log(1+R^2))^{j-1}\right.\cut\left.+IS^0(R^{-2})(\log(1+R^2))^{j-2}\right]\cut
+a^{-1}\partial_aW_{2k}^jIS^0(R^{-2})(\log(1+R^2))^{j-1}\right\}.
\end{multline}
Using \eref{wa3q} it follows, using also the argument from equations \eref{w1}--\eref{w2} as well as basic properties of the $IS$ spaces, that
\begin{multline}\label{step4-2}
t^2\left(e_{2k-1}^0+ \left( -\partial_t^2+\opr-\frac{1}{r^2}\right)v_{2k}\right)\\\in\tl{2k}\isqp{1}{-1}{2k-2}{k}\\
\subset\frac{1}{(t\lambda)^{2k}}\left[\isq{1}{-1}{2k}{k}+\!\!\!\!\!
\sum_{\beta=b,b_1,b_2}\beta\,\isqp{1}{}{2k-1}{k}\right]
\end{multline}

\subsubsection{The third term of \eref{def-e2k}}

Write first
\begin{multline}\label{n2k}
-t^2N_{2k}(v_{2k})=a^{-2}\left[\left(f(u_{2k-1}+v_{2k})-f(u_{2k-1})-f'(u_{2k-1})v_{2k}\right)\cut+
\left(f'(u_{2k-1})-f'(u_0)\right)v_{2k}+\left(f'(u_0)-1\right)v_{2k}\right]\\=a^{-2}[I+II+III]
\end{multline}

Now
\begin{equation}
I=v_{2k}^2\sum_{l\geq2}\frac{1}{l!}f^{(l)}(u_{2k-1})v_{2k}^{l-2}.
\end{equation}
Remembering the computation \eref{eq-for-lemma-f-2},
\begin{equation}
v_{2k}\in\tl{2}\isq{1}{}{}{},\quad v_{2k}^2\in IS^0(1,\Q).
\end{equation}
By Lemma \ref{useful-lemma-2}
\begin{equation}
f^{(2m)}(u_{2k-1})v_{2k}^{2m-2}\in\tl{2}\isq{1}{}{}{}
\end{equation}
and
\begin{equation}
f^{(2m+1)}(u_{2k-1})v_{2k}v_{2k}^{2m-2}\in\tl{2}\isq{1}{}{}{}.
\end{equation}
Therefore
\begin{multline}\label{step4-3}
I\in v_{2k}^2\tl{2}\isq{1}{}{}{}
\subset\tl{4k+6}\isq{7}{7}{4k-1}{k}\\
\subset\tl{2k+6}\sum_{\beta=b,b_1,b_2}\beta\isq{7}{7}{2k-1}{k}\\
\subset a^6\tl{2k}\sum_{\beta=b,b_1,b_2}\beta\isqp{1}{}{2k-1}{k}.
\end{multline}

The second term in \eref{n2k} can be written as
\begin{equation}
II=v_{2k}\sum_{l\geq2}\frac{1}{(l-1)!}f^{(l)}(u_0)(u_{2k-1}-u_0)^{l-1}.
\end{equation}
Recall that
\begin{equation}
u_{2k-2}-u_0\in\tl{2}\isq{1}{}{}{},\quad(u_{2k-2}-u_0)^2\in IS^0(1,\Q).
\end{equation}
Using Lemma \ref{useful-lemma-1},
\begin{multline}
f^{(2m)}(u_0)(u_{2k-2}-u_0)^{2m-1}\in\tl{2}IS^2(\log R,\Q)\\\subset\tl{2}\sum_{\beta=b,b_1,b_2}\beta IS^2(R^2)
\subset a^2 \sum_{\beta=b,b_1,b_2}\beta IS^0(1,\Q)
\end{multline}
and
\begin{equation}
f^{(2m+1)}(u_0)(u_{2k-2}-u_0)^{2m}\in a^2 \sum_{\beta=b,b_1,b_2}\beta IS^0(1,\Q).
\end{equation}
Then
\begin{multline}\label{step4-4}
II\in\tl{2k+2}\sum_{\beta=b,b_1,b_2}\beta \isq{3}{3}{2k-1}{k}\\
\subset a^2 \tl{2k}\sum_{\beta=b,b_1,b_2}\beta \isqp{1}{}{2k-1}{k}.
\end{multline}

The last term in \eref{n2k} is
\begin{equation}
III=v_{2k}\left(f'(u_0)-1\right)\in v_{2k}IS^2\left(R^{-2}\right),
\end{equation}
therefore
\begin{multline}\label{step4-5}
III\in \tl{2k+2}\isq{5}{}{2k-1}{k}\\
\subset a^2\tl{2k}\isq{3}{-1}{2k}{k}.
\end{multline}

Putting together the results of equations \eref{step4-1}, \eref{step4-2}, \eref{step4-3}, \eref{step4-4}, and \eref{step4-5}, it follows that
\begin{align}
t^2e_{2k}&\in\frac{1}{(t\lambda)^{2k}}\left[\isq{1}{-1}{2k}{k}+b\isqp{1}{}{2k-1}{k}+\right.\nonumber\\
&\left.+b_1\isqp{1}{}{2k-1}{k}+b_2\isqp{1}{}{2k-1}{k}\right].
\end{align}

By induction, \eref{induction-1}, \eref{induction-2}, \eref{induction-3}, and \eref{induction-4} are now proved for any $k$.

\section{The Perturbed Equation}\label{pe-eq}

For a fixed $k$ define $\epsilon(t,r)$ to be such that
\begin{equation}
u(t,r)=u_{2k-1}(t,r)+\epsilon(t,r),
\end{equation}
where $u$ is the solution of \eref{wm} that is being constructed. Then $\epsilon$ needs to solve the following equation 
\begin{equation}\label{equation-epsilon}
-\partial_t^2\epsilon+\partial_r^2\epsilon+\frac{1}{r}\partial_r\epsilon-\frac{f'(u_0)}{r^2}\epsilon=-e_{2k-1}-N_{2k-1}(\epsilon),
\end{equation}
where
\begin{equation}
N_{2k-1}(\epsilon)=\frac{1}{r^2}\left[f'(u_0)\epsilon-f(u_{2k-1}+\epsilon)-f(u_{2k-1})\right].
\end{equation}

If the time variable is replaced by $\tau=\frac{1}{\nu}t^{-\nu}$, the space varialble by $R=\lambda(t)r$, and with the notation 
$v(\tau,R)=\epsilon(t,\lambda^{-1}R)$, then \eref{equation-epsilon} becomes
\begin{multline}
-\left[\left(\partial_\tau+\frac{\lambda_\tau}{\lambda}R\partial_R\right)^2+
\frac{\lambda_\tau}{\lambda}\left(\partial_\tau+\frac{\lambda_\tau}{\lambda}R\partial_R\right)\right]v+
\left(\partial_R^2+\frac{1}{R}\partial_R-\frac{f'(Q(R))}{R^2}\right)v\\=
-\frac{1}{\lambda^2}\left[N_{2k-1}(\epsilon)+e_{2k-1}\right].
\end{multline}
After making the further change of function $\ept(\tau,R)=R^{1/2}v(\tau,R)$, \eref{equation-epsilon} becomes
\begin{multline}\label{equation-epsilon-tilde}
\left[-\left(\partial_\tau+\frac{\lambda_\tau}{\lambda}R\partial_R\right)^2+\frac{1}{4}\left(\frac{\lambda_\tau}{\lambda}\right)^2+
\frac{1}{2}\partial_\tau\left(\frac{\lambda_\tau}{\lambda}\right)\right]\ept-\el\ept\\=
-\lambda^{-2}R^{1/2}\left[N_{2k-1}(R^{-1/2}\ept)+e_{2k-1}\right],
\end{multline}
where
\begin{equation}
\el=-\partial_R^2+\frac{3}{4R^2}+V(R),\quad\quad V(R)=-\frac{1}{R^2}\left[1-f'(Q(R))\right].
\end{equation}
This last change of function has the benefit that it produces $\el$, which is  a self-adjoint operator on $L^2(\R^+,dR)$.

\section{The Transference Identity}\label{tr-id}

The plan to deal with \eref{equation-epsilon-tilde} is to expand $\ept$ in terms of the generalized Fourier basis $\phi(R,\xi)$ of the operator $\el$ 
(see Theorem \ref{tgz}):
\begin{equation}
\ept(\tau,R)=\int_0^\infty x(\tau,\xi)\phi(R,\xi)\rho(\xi)\;d\xi.
\end{equation}
The coefficinets $x(\tau,\xi)$ would then hopefully satisfy a transport equation. However, $R\partial_R$ is not diagonal in this Fourier basis. To deal with this, $R\partial_R$ will be replaced by $2\xi\partial_\xi$ and the error will be treated as a perturbation. 

This section follows closely section 6 of \cite{KST}, to the point of being identical. This is due to the fact that the estimates of Appendix 
\ref{appendix-A} are identical to the ones in section 5 of the reference. The main result of the section is Proposition \ref{proposition-K-mapping}, whose proof is omitted as it is identical to the proof of Proposition 6.2 in \cite{KST}.

Let the operator $\K$ be defined by\footnote{This is what is referred to as a ``transference identity''.}
\begin{equation}
\widehat{R\partial_R u}=-2\xi\partial_\xi\widehat u+\K\widehat u,
\end{equation}
where $\widehat f=\mathcal{F}f$ is the distorted Fourier transform defined in Theorem \ref{tgz}. Using the definitions for this Fourier transform and its inverse, $\K$ can be written as
\begin{multline}
\K f(\eta)=\left<\int_0^\infty f(\xi)R\partial_R\phi(R,\xi)\rho(\xi)\,d\xi,\;\phi(R,\eta)\right>_{L_R^2}\\+
\left<\int_0^\infty 2\xi\partial_\xi f(\xi)\phi(R,\xi)\rho(\xi)\,d\xi,\;\phi(R,\eta)\right>_{L_R^2}.
\end{multline}
Integrating by parts with respect to $\xi$,
\begin{multline}\label{Kf}
\K f(\eta)=\left<\int_0^\infty f(\xi)\left[R\partial_R-2\xi\partial_\xi\right]\phi(R,\xi)\rho(\xi)\,d\xi,\;\phi(R,\eta)\right>_{L_R^2}\\-
2\left(1+\frac{\eta\rho'(\eta)}{\rho(\eta)}\right)f(\eta).
\end{multline}
The scalar product is interpreted in the principal value sense with $f\in C_0^\infty(0,\infty)$. A priori
\begin{equation}
\K:C_0^\infty(0,\infty)\to C^\infty(0,\infty),
\end{equation}
therefore there is a distribution valued function $\eta\to K(\eta,\xi)$ such that
\begin{equation}
\K f(\eta)=\int_0^\infty k(\eta,\xi)f(\xi)\;d\xi.
\end{equation}

\begin{thm}
The operator $\K$ can be written as
\begin{equation}
\K=-\left(\frac{3}{2}+\frac{\eta\rho'(\eta)}{\rho(\eta)}\right)\delta(\xi-\eta)+\K_0,
\end{equation}
where the operator $\K_0$ has a kernel $K_0(\eta,\xi)$ of the form (in the principal value sense):
\begin{equation}\label{K-zero}
K_0(\eta,\xi)=\frac{\rho(\xi)}{\xi-\eta}F(\xi,\eta),
\end{equation}
with a symmetric function $F(\xi,\eta)$ of class $C^2$ in $(0,\eta)\times(0,\eta)$ satisfying the bounds
\begin{gather}
|F(\xi,\eta)|\lesssim
\left\{\!\!\!\begin{array}{cl} \xi+\eta&\xi+\eta\leq1\\ (\xi+\eta)^{-3/2}(1+|\xi^{1/2}-\eta^{1/2}|)^{-N}&\xi+\eta\geq1\end{array}\!\!\!\right.,\\
\!\!\!\!\!\!|\partial_\xi F(\xi,\eta)|+|\partial_\eta F(\xi,\eta)|
\lesssim\left\{\!\!\!\begin{array}{cl} 1&\xi+\eta\leq1\\ (\xi+\eta)^{-2}(1+|\xi^{1/2}-\eta^{1/2}|)^{-N}&\xi+\eta\geq1\end{array}\!\!\!\right.,\\
\sup_{j+k=2}|\partial_\xi^j\partial_\eta^k F(\xi,\eta)|
\lesssim\left\{\!\!\!\begin{array}{cl} |\log(\xi+\eta)|^3&\xi+\eta\leq1\\ (\xi+\eta)^{-5/2}(1+|\xi^{1/2}-\eta^{1/2}|)^{-N}&\xi+\eta\geq1\end{array}\!\!\!\right.,
\end{gather}
where $N$ is an arbitrary large integer.
\end{thm}

\begin{proof}
The off-diagonal behavior of $K$ is addressed first. Let $f\in C_0^\infty(0,\infty)$. Then
\begin{equation}
u(R)=\int_0^\infty f(\xi)[R\partial_R-2\xi\partial_\xi]\phi(R,\xi)\rho(\xi)\;d\xi
\end{equation}
behaves like $R^{3/2}$ at $0$ and like a Schwartz function at infinity. The second factor in \eref{Kf}, $\phi(R,\eta)$, decays like $R^{3/2}$ at zero, but at infinity is bounded, with bounded derivatives. Using integration by parts:
\begin{equation}
\eta\K f(\eta)=\left<u,\el\phi(R,\eta)\right>_{L_R^2}=\left<\el u,\phi(R,\eta)\right>_{L_R^2}.
\end{equation}
Moreover,
\begin{multline}
\el u=\int_0^\infty f(\xi)[\el,R\partial_R]\phi(R,\xi)\rho(\xi)\;d\xi+\int_0^\infty f(\xi)(R\partial_R-2\xi\partial_\xi)\xi\phi(R,\xi)\rho(\xi)\;d\xi\\
=\int_0^\infty f(\xi)[\el,R\partial_R]\phi(R,\xi)\rho(\xi)\;d\xi+\int_0^\infty \xi f(\xi)(R\partial_R-2\xi\partial_\xi)\phi(R,\xi)\rho(\xi)\;d\xi\\
-2\int_0^\infty\xi f(\xi)\phi(R,\xi)\rho(\xi)\;d\xi,
\end{multline}
with the comutator
\begin{equation}
[\el,R\partial_R]=2\el-2\left(V(R)+RV'(R)\right)=2\el+W(R).
\end{equation}
Thus
\begin{equation}
\el u=\!\!\int_0^\infty \!\!\!\!f(\xi)W(R)\phi(R,\xi)\rho(\xi)\;d\xi+\!\!\int_0^\infty \!\!\!\!\xi f(\xi)(R\partial_R-2\xi\partial_\xi)\phi(R,\xi)\rho(\xi)\;d\xi.
\end{equation}
Hence
\begin{equation}
\eta\K f(\eta)-\K(\xi f)(\eta)=\left<\int_0^\infty f(\xi)W(R)\phi(R,\xi)\rho(\xi)\;d\xi,\phi(R,\eta)\right>_{L_R^2}.
\end{equation}
Changing the order of integration on the right hand side yields:
\begin{equation}
(\eta-\xi)K(\eta,\xi)=\rho(\xi)\left<W(R)\phi(R,\xi),\phi(R,\eta)\right>_{L_R^2}.
\end{equation}
This gives the representation \eref{K-zero} when $\eta\neq\xi$, with 
\begin{equation}
F(\xi,\eta)=\left<W(R)\phi(R,\xi),\phi(R,\eta)\right>_{L_R^2}.
\end{equation}
It remains to study its size and regularity. By Proposition \ref{phi-expansion-proposition},
\begin{gather}
\sup_{R\geq0}|\phi(R,\xi)|\lesssim<\xi>^{-3/4},\label{phi-poinwise-bounds-a}\\
|R\partial_R\phi(R,\xi)|\lesssim\min(R\xi^{-1/4},R^{3/2}),\quad\forall\xi>1,\\
|\partial_\xi\phi(R,\xi)|\lesssim\min(R\xi^{-5/4},R^{7/2}),\quad\forall\xi>1/2,\\
|\partial_\xi\phi(R,\xi)|\lesssim\min(R^{3/2}\log(1+R^2),\xi^{-1/4}|\log\xi|R),\quad\forall0<\xi<1/2,\\
|\partial_\xi^2\phi(R,\xi)|\lesssim\min(R^2\xi^{-7/4},R^{11/2}),\quad\forall\xi>1/2,\\
|\partial_\xi^2\phi(R,\xi)|\lesssim\min(R^{7/2}\log(1+R^2),\xi^{-3/4}|\log\xi|R^2),\quad\forall0<\xi<1/2,\label{phi-poinwise-bounds-b}
\end{gather}
therefore
\begin{gather}
|F(\xi,\eta)|\lesssim<\xi>^{-3/4}<\eta>^{-3/4},\label{F-poinwise-bounds-a}\\
|\partial_\xi F(\xi,\eta)|\lesssim<\xi>^{-5/4}<\eta>^{-3/4},\\
|\partial_\eta F(\xi,\eta)|\lesssim<\xi>^{-3/4}<\eta>^{-5/4},\\
|\partial_{\xi\eta}^2 F(\xi,\eta)|\lesssim<\xi>^{-5/4}<\eta>^{-5/4},\quad\xi+\eta\gtrsim1,\\
|\partial_{\xi}^2 F(\xi,\eta)|\lesssim<\xi>^{-7/4}<\eta>^{-3/4},\quad\xi>1,\eta>1,\\
|\partial_{\eta}^2 F(\xi,\eta)|\lesssim<\xi>^{-7/4}<\eta>^{-3/4},\quad\xi>1,\eta>1.\label{F-poinwise-bounds-b}
\end{gather}
To improve on these, two cases will be considered.

\paragraph{Case 1:} $1\lesssim\xi+\eta$. By integration by parts:
\begin{multline}
\eta F(\xi,\eta)=\left< W(R)\phi(R,\xi),\el\phi(R,\eta)\right>_{L_R^2}\\=\left<[\el,W(R)]\phi(R,\xi),\phi(R,\eta)\right>_{L_R^2}+\xi F(\xi,\eta).
\end{multline}
Evaluating the commutator:
\begin{equation}\label{eta-xi-F}
(\eta-\xi)F(\xi,\eta)=-\left<(2W'\partial_R+W'')\phi(R,\xi),\phi(R,\eta)\right>_{L_R^2}.
\end{equation}
Since $W'(0)=0$ (it is odd), it follows that  $(2W'\partial_R+W'')\phi(R,\xi)$ has the same behavior as $\phi(R,\xi)$ at $R=0$. Then the argument can be repeated to obtain:
\begin{equation}
(\eta-\xi)^2F(\xi,\eta)=-\left<[\el,2W'\partial_R+W'']\phi(R,\xi),\phi(R,\eta)\right>_{L_R^2}.
\end{equation}
This second commutator has the form:
\begin{multline}
[\el,2W'\partial_R+W'']=4W''\el-4W'''\partial_R-W^{(4)}\\+3R^{-2}(R^{-1}W'-W'')-2W'V'-4W''V.
\end{multline}
Since $R^{-1}W'(R)-W''(R)=\oh(R^2)$, this leads to
\begin{equation}
(\eta-\xi)^2F(\xi,\eta)=\left<(W^{o}(R)\partial_R+W^{e}(R)+\xi W^{e}(R))\phi(R,\xi),\phi(R,\eta)\right>_{L_R^2},
\end{equation}
where $W^{o}$, respectively $W^{e}$, are  odd, respectively even, real-analytic functions with good decay at infinity. Inductively
\begin{multline}\label{F-induction}
(\eta-\xi)^{2k}F(\xi,\eta)\\
=\left<\left(\sum_{j=0}^{k-1}\xi^j W_{kj}^o(R)\partial_R+\sum_{l=0}^k\xi^lW_{kl}^e(R)\right)\phi(R,\xi),\phi(R,\eta)\right>_{L_R^2},
\end{multline}
where
\begin{equation}
<R>|W_{kj}^o(R)|+|W_{kl}^e|\lesssim<R>^{-4-2k},\quad\forall j,l.
\end{equation}
Using the pointwise bounds on $\phi$ and $\partial_R\phi$ from \eref{phi-poinwise-bounds-a}--\eref{phi-poinwise-bounds-b}:
\begin{equation}\label{F-bound}
|F(\xi,\eta)|\lesssim\frac{\xi^{k-3/4}<\eta>^{-3/4}}{(\eta-\xi)^{2k}},\quad\forall\xi\gtrsim1, \eta>0.
\end{equation}
Combining this with \eref{F-poinwise-bounds-a}--\eref{F-poinwise-bounds-b}, it yields, for arbitrary $N$, that
\begin{equation}
|F(\xi,\eta)|\lesssim(\xi+\eta)^{-3/2}(1+|\xi^{1/2}-\eta^{1/2}|)^{-N},\quad\text{if }\xi+\eta\gtrsim1.
\end{equation}

For the derivatives of $F$ a similar procedure can be used. If $\xi$ and $\eta$ are comparable, then from 
\eref{F-poinwise-bounds-a}--\eref{F-poinwise-bounds-b}
\begin{equation}
|\partial_\eta F(\xi,\eta)|\lesssim<\xi>^{-2}.
\end{equation}
Otherwise, differentiating with respect to $\eta$ in \eref{F-induction},
\begin{multline}
(\eta-\xi)^{2k}\partial_\eta F(\xi,\eta)\\
=\left<\left(\sum_{j=0}^{k-1}\xi^j W_{kj}^o(R)\partial_R+\sum_{l=0}^k\xi^lW_{kl}^e(R)\right)\phi(R,\xi),\partial_\eta\phi(R,\eta)\right>_{L_R^2}\\-
2k(\eta-\xi)^{2k-1}F(\xi,\eta).
\end{multline}
Using also \eref{F-bound}, it follows that
\begin{equation}
|\partial_\eta F(\xi,\eta)|\lesssim\frac{\xi^{k-3/4}\eta^{-5/4}}{(\eta-\xi)^{2k}},\quad 1\lesssim\xi,\eta,
\end{equation}
respectively
\begin{equation}
|\partial_\eta F(\xi,\eta)|\lesssim\frac{\eta^{-5/4}}{(\eta-\xi)^{2k}},\quad \xi\ll1\lesssim\eta,
\end{equation}
and
\begin{equation}
|\partial_\eta F(\xi,\eta)|\lesssim\frac{\xi^{k-3/4}}{(\eta-\xi)^{2k}},\quad \eta\ll1\lesssim\xi,
\end{equation}
which yield the desired bounds.

Finally, consider the second order derivatives with respect to $\xi$ and $\eta$. For $\xi$ and $\eta$ close, 
\eref{F-poinwise-bounds-a}--\eref{F-poinwise-bounds-b} can be used. Otherwise, differentiate twice in \eref{F-induction} and continue as before. Note that it is important that the decay of $W_{kj}^o$ and $W_{kj}^e$ improves with $k$. This is because the second order derivative bound at zero has a sizable growth at infinity which has to be canceled,
\begin{equation}
|\partial_\xi^2\phi(R,0)|\approx R^{7/2}\log R.
\end{equation}

\paragraph{Case 2:} $\xi,\eta\ll1$. First note that $F(0,0)=0$. This can be verified by direct computation. Also by direct computation it can be checked that
\begin{equation}
|\partial_\xi F(\xi,\eta)\lesssim1
\end{equation}

To obtain the bound on the second derivatives, begin by observing that the following inequalities hold:
\begin{equation}
|\partial_\xi^j\phi(R,\xi)|\lesssim\left\{\begin{array}{cl}
R^{-1/2+2j}\log(1+R^2)\quad&R<\xi^{-1/2}\\ & \\
\xi^{1/4-j/2}|\log\xi|R^j&R\geq\xi^{-1/2}
\end{array}\right.,\quad j=0,1,2.
\end{equation}
If $\eta<\xi<1/2$, then these bounds imply that
\begin{multline}
|\partial_{\xi\eta}^2F(\xi,\eta)|\lesssim\int_0^{\xi^{-1/2}}<R>^{-4}R^3(\log(1+R^2))^2\;dR\\+
\int_{\xi^{-1/2}}^{\eta^{-1/2}}<R>^{-4}R^{5/2}\xi^{-1/4}|\log\xi|\log(1+R^2)\;dR\\+
\int_{\eta^{-1/2}}^\infty <R>^{-2}\xi^{-1/4}\eta^{-1/4}|\log\xi|\,|\log\eta|\;dR\lesssim|\log\xi|^3.
\end{multline}
The main contribution comes from the first term. When $\eta<\xi<1/2$, a similar computation yields
\begin{multline}
|\partial_{\xi}^2F(\xi,\eta)|\lesssim\int_0^{\xi^{-1/2}}<R>^{-4}R^3(\log(1+R^2))^2\;dR\\+
\int_{\xi^{-1/2}}^{\eta^{-1/2}}<R>^{-4}R^{3/2}\xi^{-3/4}|\log\xi|\log(1+R^2)\;dR\\+
\int_{\eta^{-1/2}}^\infty <R>^{-2}\xi^{-3/4}\eta^{1/4}|\log\xi|\,|\log\eta|\;dR\lesssim|\log\xi|^3.
\end{multline}
It remains to consider $\partial_\xi^2F(\xi,\eta)$ when $\xi\ll\eta<1/2$. Differentiating \eref{eta-xi-F},
\begin{equation}
(\eta-\xi)\partial_\xi^2F(\xi,\eta)=2\partial_\xi F(\xi,\eta)-
\left<\partial_\xi^2\phi(R,\xi),(2W'\partial_R+W'')\phi(R,\eta)\right>_{L_R^2}.
\end{equation}
Differentiating and integrating with respect to $\eta$
\begin{multline}\label{xi-eta-zeta-F}
(\eta-\xi)\partial_\xi^2F(\xi,\eta)\\
=\int_\xi^\eta\left[2\partial^2_{\xi\zeta} F(\xi,\zeta)-\left<\partial_\xi^2\phi(R,\xi),(2W'\partial_R+W'')\partial_\zeta\phi(R,\zeta)\right>_{L_R^2}\right]\;d\zeta.
\end{multline}
Using the bound
\begin{equation}
|\partial_R\partial_\zeta\phi(R,\zeta)|\lesssim\left\{\begin{array}{cl} R^{1/2}\log(1+R^2)\quad&R<\zeta^{-1/2}\\ \\
\zeta^{-1/4}|\log\zeta|&R\geq\zeta^{-1/2}\end{array}\right.,
\end{equation}
the inner product in \eref{xi-eta-zeta-F} can be evaluated as follows:
\begin{multline}
\left| \left<\partial_\xi^2\phi(R,\xi),(2W'\partial_R+W'')\partial_\zeta\phi(R,\zeta)\right>_{L_R^2} \right|\\\lesssim
\int_0^{\zeta^{-1/2}}<R>^{-6}R^{7/2}\log(1+R^2)R^{3/2}\log(1+R^2)\;dR\\+
\int_{\zeta^{-1/2}}^{\xi^{-1/2}}<R>^{-6}R^{7/2}\log(1+R^2)\zeta^{-1/4}|\log\zeta|R\;dR\\+
\int_{\xi^{-1/2}}^\infty<R>^{-6}\xi^{-3/4}|\log\xi|R^2\zeta^{-1/4}|\log\zeta|R\;dR\lesssim|\log\zeta|^3.
\end{multline}
Thus, \eref{xi-eta-zeta-F} is controlled by 
\begin{equation}
|(\eta-\xi)\partial_\xi^2F(\xi,\eta)|\lesssim\left|\int_\xi^\eta(\log \zeta)^3\;d\zeta \right|\lesssim\eta|\log\eta|^3.
\end{equation}
Since $\xi\ll\eta$, this yields
\begin{equation}
|\partial_\eta^2F(\xi,\eta)|\lesssim|\log\eta|^3.
\end{equation}
This concludes the analysis of the off-diagonal part of the kernel.

All that is left now is to determine the $\delta$ measure that sits on the diagonal of the kernel $K$. To do so, first restrict $\xi$ and $\eta$ to a compact set of $(0,\infty)$. Then the following asymptotics hold for $R\xi^{1/2}\gg1$:
\begin{equation}
\phi(R,\xi)=\re\left[a(\xi)\xi^{-1/4}e^{iR\xi^{1/2}}\left(1+\frac{3i}{8R\xi^{1/2}}\right)\right]+\oh(R^{-2}),
\end{equation}
\begin{multline}
(R\partial_R-2\xi\partial_\xi)\phi(R,\xi)\\=-2\re\left[\xi\partial_\xi(a(\xi)\xi^{-1/4})e^{iR\xi^{1/2}}\left(1+\frac{3i}{8R\xi^{1/2}}\right)\right]+\oh(R^{-2}),
\end{multline}
where the $\oh$ terms depend on the choice of compact subset. The $R^{-2}$ terms are integrable, so they contribute a bounded kernel to the inner product in \eref{Kf}. The same applies to the contribution of  a bounded $R$ region. Therefore, the $\delta$-measure contribution of the inner product in \eref{Kf} can only come from one of the following integrals:
\begin{multline}\label{integral-delta-a}
-\int_0^\infty\int_0^\infty f(\xi)\chi(R)\re\left[\xi\partial_\xi(a(\xi)\xi^{-1/4})a(\eta)\eta^{-1/4}e^{iR(\xi^{1/2}+\eta^{1/2})}\cut\times
\left(1+\frac{3i}{8R\xi^{1/2}}\right)\left(1+\frac{3i}{8R\eta^{1/2}}\right)\right]\rho(\xi)d\xi dR,
\end{multline}
\begin{multline}\label{integral-delta-b}
-\frac{1}{2}\int_0^\infty\int_0^\infty f(\xi)\chi(R)\xi\partial_\xi(a(\xi)\xi^{-1/4})\overline{a}(\eta)\eta^{-1/4}e^{iR(\xi^{1/2}-\eta^{1/2})}\\\times
\left(1+\frac{3i}{8R\xi^{1/2}}\right)\left(1-\frac{3i}{8R\eta^{1/2}}\right)\rho(\xi)d\xi dR,
\end{multline}
\begin{multline}\label{integral-delta-c}
-\frac{1}{2}\int_0^\infty\int_0^\infty f(\xi)\chi(R)\xi\partial_\xi(\overline{a}(\xi)\xi^{-1/4})a(\eta)\eta^{-1/4}e^{-iR(\xi^{1/2}-\eta^{1/2})}\\\times
\left(1-\frac{3i}{8R\xi^{1/2}}\right)\left(1+\frac{3i}{8R\eta^{1/2}}\right)\rho(\xi)d\xi dR,
\end{multline}
where $\xi$ is a smooth cutoff function which equals $0$ near $R=0$ and $1$ near $R=\infty$. In all of the above integrals it can be argued, as in the proof of the classical Fourier inversion formula, that the order of integration can be changed. Integration by parts in the first integral 
\eref{integral-delta-a} reveals that it cannot contribute to the $\delta$-measure. Discarding the $\oh(R^{-2})$ terms in \eref{integral-delta-b} and 
\eref{integral-delta-c} reduces the two integrals to:
\begin{equation}\label{integral-delta-d}
-\int_0^\infty\!\!\!\!\!\int_0^\infty \!\!\!\!\!\!f(\xi)\chi(R)\re\left[\xi\partial_\xi(a(\xi)\xi^{-1/4})\overline{a}(\eta)\eta^{-1/4}e^{iR(\xi^{1/2}-\eta^{1/2})}\right]
\rho(\xi)d\xi dR,
\end{equation}
\begin{multline}\label{integral-delta-e}
+\frac{3}{8}\int_0^\infty\int_0^\infty f(\xi)\chi(R)
\im\left[\xi\partial_\xi(a(\xi)\xi^{-1/4})\overline{a}(\eta)\eta^{-1/4}e^{iR(\xi^{1/2}-\eta^{1/2})}\right]\\\times
R^{-1}(\xi^{-1/2}-\eta^{-1/2})\rho(\xi)d\xi dR,
\end{multline}
Since \eref{integral-delta-e} contains both  an $R^{-1}$ and a $(\xi^{-1/2}-\eta^{-1/2})$ factor, its contribution to $K$ is bounded. The integral 
\eref{integral-delta-d} contributes both a Hilbert transform type kernel as well as a $\delta$-measure to $K$. By inspection, the $\delta$-measure contribution is:
\begin{multline}
-\frac{1}{2}\int_{-\infty}^\infty \re\left[\xi\partial_\xi(a(\xi)\xi^{-1/4})\overline{a}(\eta)\eta^{-1/4}e^{iR(\xi^{1/2}-\eta^{1/2})}\right]\rho(\xi)dR\\
=-\pi\re\left[\xi\partial_\xi(a(\xi)\xi^{-1/4})\overline{a}(\eta)\eta^{-1/4}\right]\rho(\xi)\delta(\xi^{1/2}-\eta^{1/2})\\
=-2\pi\xi^{1/2}\rho(\xi)\re\left[\xi\partial_\xi(a(\xi)\xi^{-1/4})\overline{a}(\xi)\xi^{-1/4}\right]\delta(\xi-\eta)\\
=-2\pi\xi^{1/2}\rho(\xi)\re\left[\frac{1}{4}\xi^{-1/2}|a(\xi)|^2+\xi^{1/2}a(\xi)\overline{a}'(\xi)\right]\delta(\xi-\eta)\\
=\left[\frac{1}{2}+\frac{\xi\rho'(\xi)}{\rho(\xi)}\right]\delta(\xi-\eta),
\end{multline}
where the fact that $\rho(\xi)^{-1}=\pi|a|^2$ was used in the last step. This finishes the proof.
\end{proof}

The following proposition establishes some $L^2$ mapping properties of $\K$. Since the conclusion of the preceeding theorem and 
the results of appendix \ref{appendix-A} are the same as their correspondents in \cite{KST}, the proof of this result is omitted as it is identical to the proof of Proposition 6.2 in the reference. 

First let $\lr{\alpha}$ be the $L^2$ space with the norm
\begin{equation}
\nlr{f}{\alpha}=\left(\int_0^\infty|f(\xi)|^2<\xi>^{2\alpha}\rho(\xi)\;d\xi\right)^{1/2}.
\end{equation}
Then
\begin{pp}\label{proposition-K-mapping}
\begin{itemize}
\item[i)] The operator $\K_0$ maps
\begin{equation}
\K_0:\lr{\alpha}\to\lr{\alpha+1/2};
\end{equation}
\item[ii)] In addition, the following commutator bound holds:
\begin{equation}
[\K_0,\xi\partial_\xi]:\lr{\alpha}\to\lr{\alpha}.
\end{equation}
\end{itemize}
Both statements hold for al $\alpha\in\R$. In particular, $\K$ and $[\K,\xi\partial_\xi]$ are bounded operators on $\lr{\alpha}$.
\end{pp}

\section{The Final Equation}\label{fi-eq}

To rewrite \eref{equation-epsilon-tilde} in a final form, begin by expressing the operator $R\partial_R$ in terms of $\K$. Therefore, with $\mathcal{F}$ as in Theorem \ref{tgz},
\begin{equation}
\mathcal{F}\left(\partial_\tau+\frac{\lambda_\tau}{\lambda}R\partial_R\right)=
\left(\partial_\tau+\frac{\lambda_\tau}{\lambda}(-2\xi\partial_\xi+\K)\right)\mathcal{F},
\end{equation}
which gives
\begin{multline}
\mathcal{F}\left(\partial_\tau+\frac{\lambda_\tau}{\lambda}R\partial_R\right)^2=
\left(\partial_\tau+\frac{\lambda_\tau}{\lambda}(-2\xi\partial_\xi+\K)\right)^2\mathcal{F}\\=
\left(\partial_\tau-\frac{\lambda_\tau}{\lambda}2\xi\partial_\xi\right)^2\mathcal{F}+
2\frac{\lambda_\tau}{\lambda}\K\left(\partial_\tau-\frac{\lambda_\tau}{\lambda}2\xi\partial_\xi\right)\mathcal{F}\\+
\frac{\lambda_\tau^2}{\lambda^2}\left(\K^2+2[\xi\partial_\xi,\K]\right)  \mathcal{F}.
\end{multline}
This leads to a transport type equation for the Fourier transform $x(\tau,\xi)$ of $\widetilde\epsilon$:
\begin{multline}\label{equation-x}
-\left(\partial_\tau-\frac{\lambda_\tau}{\lambda}2\xi\partial_\xi\right)^2x-\xi x=
2\frac{\lambda_\tau}{\lambda}\K\left(\partial_\tau-\frac{\lambda_\tau}{\lambda}2\xi\partial_\xi\right)x\\+
\frac{\lambda_\tau^2}{\lambda^2}\left(\K^2+2[\xi\partial_\xi,\K]\right)x-
\left(\frac{1}{4}\left(\frac{\lambda_\tau}{\lambda}\right)^2+\frac{1}{2}\partial_\tau\left(\frac{\lambda_\tau}{\lambda}\right)\right)x\\+
\lambda^{-2}\mathcal{F}R^{1/2}\left(N_{2k-1}(R^{-1/2}\mathcal{F}^{-1}x)+e_{2k-1} \right).
\end{multline}

The aim is to obtain solutions of \eref{equation-x} which decay as $\tau\to\infty$. This means the equation will be solved backwards in time, with zero Cauchy data at $\tau=\infty$. The problem will be treated iteratively, as a small perturbation of the linear equation governed by the operator on the left-hand side. For this the following transport equation needs to be solved:
\begin{equation}\label{transport-equation}
-\left[\left(\partial_\tau-\frac{\lambda_\tau}{\lambda}2\xi\partial_\xi\right)^2+\xi\right]x(\tau,\xi)=b(\tau,\xi).
\end{equation}
Denote by $H$ the backward fundamental solution of the operator
\begin{equation}
\left(\partial_\tau-\frac{\lambda_\tau}{\lambda}2\xi\partial_\xi\right)^2+\xi,
\end{equation}
and by $H(\tau,\sigma)$ its kernel, i.e. \eref{transport-equation} has solution
\begin{equation}
x(\tau)=-\int_\tau^\infty H(\tau,\sigma)b(\sigma)\;d\sigma,
\end{equation}
where the $\xi$ variable has been suppressed. The mapping properties of $H$ are described in the following result, which is proven in \cite{KST}, section 8.
\begin{pp}\label{p-transport-1}
For any $\alpha\geq0$ there exists some (large) constant $C=C(\alpha)$ so that the operator $H(\tau,\sigma)$ satisfies the bounds
\begin{equation}
\left\lVert H(\tau,\sigma)\right\rVert_{\lr{\alpha}\to\lr{\alpha+1/2}}\lesssim\tau\left(\frac{\sigma}{\tau}\right)^C,
\end{equation}
\begin{equation}
\left\lVert\left(\partial_\tau-\frac{\lambda_\tau}{\lambda}2\xi\partial_\xi\right)H(\tau,\sigma)\right\rVert_{\lr{\alpha}\to\lr{\alpha}}\lesssim
\tau\left(\frac{\sigma}{\tau}\right)^C,
\end{equation}
uniformly in $\sigma\geq\tau$.
\end{pp}
This leads to the introduction of the spaces $\llr{N}{\alpha}$ with norm
\begin{equation}
\nllr{f}{N}{\alpha}=\sup_{\tau\geq1}\tau^N\nlr{f(\tau)}{\alpha}.
\end{equation}
Then an immediate consequence of the above proposition is the following
\begin{cor}
Given $\alpha\geq0$, let $N$ be large enough. Then
\begin{multline}
\nllr{Hb}{N-2}{\alpha+1/2}+\nllr{\left(\partial_\tau-\frac{\lambda_\tau}{\lambda}2\xi\partial_\xi\right)Hb}{N-1}{\alpha}\\
\leq C_0 N^{-1}\nllr{b}{N}{\alpha},
\end{multline}
with a constant $C_0$ that depends on $\alpha$ but does not depend on $N$.
\end{cor}
The nonlinear operator $N_{2k-1}$ from \eref{equation-x} has the following mapping properties (which are proved below):
\begin{pp}\label{nonlinear-terms}
Assuming that $N$ is large enough and $\frac{\nu}{2}+\frac{3}{4}>\alpha>\frac{1}{4}$, then the map
\begin{equation}
x\to\lambda^{-2}\mathcal{F}\left(R^{1/2}N_{2k-1}(R^{-1/2}\mathcal{F}^{-1}x) \right)
\end{equation}
is locally Lipschitz from $\llr{N-2}{\alpha+1/2}$ to $\llr{N}{\alpha}$
\end{pp}

This two results above, combined with Proposition \ref{proposition-K-mapping} allow for the use of a contraction argument to solve equation
\eref{equation-x}.

\section{The Nonlinear Terms}\label{no-te}

The aim of this section is to prove Proposition \ref{nonlinear-terms}. First define Sobolev spaces $\hr{\alpha}$, adapted to the operator $\el$, such that
\begin{equation}
\nhr{u}{\alpha}=\nlr{\widehat u}{\alpha}.
\end{equation}
What needs to be shown is that the map
\begin{equation}
\widetilde\epsilon\to\lambda^{-2}R^{1/2}N_{2k-1}(R^{-1/2}\widetilde\epsilon)
\end{equation}
is locally Lipschitz from $\lhr{N-2}{\alpha+1/2}$ to $\lhr{N}{\alpha}$.

The following lemmas are proven in \cite{KST}:

\begin{lem}\label{n-l-1}
Let $q\in S(1,\Q)$ and $|\alpha|<\frac{\nu}{2}+\frac{3}{4}$. Then
\begin{equation}
\nhr{qf}{\alpha}\lesssim\nhr{f}{\alpha}.
\end{equation}
\end{lem}

\begin{lem}\label{n-l-2}
Let $\alpha>\frac{1}{4}$. Then
\begin{equation}
\nhr{R^{-3/2}fg}{\alpha+1/4}\lesssim\nhr{f}{\alpha+1/2}\nhr{g}{\alpha+1/2},
\end{equation}
respectively
\begin{equation}
\nhr{R^{-3/2}fg}{\alpha}\lesssim\nhr{f}{\alpha+1/4}\nhr{g}{\alpha+1/2},
\end{equation}
for all $f$, $g$ such that the right-hand sides are finite.
\end{lem}

\begin{lem}\label{n-l-3}
Let $\alpha>0$. Then 
\begin{equation}
\nhr{R^{-1}fgh}{\alpha}\lesssim\nhr{f}{\alpha+1/2}\nhr{g}{\alpha+1/2}\nhr{h}{\alpha},
\end{equation}
for all $f$, $g$, $h$ such that the right hand side is finite.
\end{lem}

Now 
\begin{multline}
R^{1/2}\lambda^{-2}N_{2k-1}(\epsilon)=-R^{-3/2}\left[\left(f'(u_{2k-1})-f'(u_0)\right)\epsilon\cut+
\left(f(u_{2k-1}+\epsilon)-f(u_{2k-1})-f'(u_{2k-1})\epsilon\right)\right]=-R^{-3/2}[I+II].
\end{multline}
For the first term write
\begin{equation}
I(\epsilon)=\epsilon\sum_{l\geq2}\frac{1}{(l-1)!}f^{(l)}(u_{2k-1}-u_0)^{l-1}
\end{equation}
Remember that $(u_{2k-1}-u_0)\in\tl{2}\isq{1}{}{}{}$ and that, by Lemma \ref{useful-lemma-1}, $f^{(2m)}(u_0)\in IS^1(R^{-1},\Q)$, 
$f^{(2m+1)}(u_0)\in IS^0(1,\Q)$. Then
\begin{multline}
f^{(2m)}(u_0)(u_{2k-1}-u_0)^{2m-1}\in\tl{4m-2}\isq{2m}{2m-2}{2m-1}{}\\
\subset\tl{4m-2}\isq{2m}{2m}{2m-2}{}\subset\tl{2m}\isq{2m}{2m}{0}{}\\
\subset\tl{2} IS^2(R^2,\Q),
\end{multline}
and
\begin{multline}
f^{(2m+1)}(u_0)(u_{2k-1}-u_0)^{2m}\in\tl{4m}\isq{2m}{2m}{2m}{}\\
\subset\tl{2m}\isq{2m}{2m}{0}{}\subset\tl{2} IS^2(R^2,\Q).
\end{multline}
Therefore
\begin{equation}
R^{-3/2}I(R^{-1/2}\widetilde\epsilon)\in\tl{2}\epsilon IS^0(1)\subset \tau^{-2}\epsilon IS^0(1).
\end{equation}
(The last step uses the fact that $t\lambda\asymp\tau$.)\footnote{By $a\asymp b$ 
it is meant that there is a positive constant $C$ such that $C^{-1}a<b<Ca$.}
So
\begin{equation}
\widetilde\epsilon\to R^{-3/2}I(R^{-1/2}\widetilde\epsilon)
\end{equation}
has the desired mapping property.

The second term can be split into two
\begin{equation}
II(\epsilon)=II_1(\epsilon)+II_2(\epsilon),
\end{equation}
where
\begin{equation}
II_1=\sum_{l\geq1}\frac{1}{(2l)!}f^{(2l)}(u_{2k-1})\epsilon^{2l}
\end{equation}
and
\begin{equation}
II_2=\sum_{l\geq1}\frac{1}{(2l+1)!}f^{(2l+1)}(u_{2k-1})\epsilon^{2l+1}.
\end{equation}
By Lemma \ref{useful-lemma-2}, 
\begin{equation}
f^{(2l)}(u_{2k-1})\in\tl{2}\isq{1}{}{}{}\subset IS^1(R).
\end{equation}
Then
\begin{equation}
R^{-3/2}f^{(2l)}(u_{2k-1})(R^{-1/2}\widetilde\epsilon)^{2l}\in (R^{-1}\widetilde\epsilon^2)^{l-1}R^{-3/2}\widetilde\epsilon^2 IS^0(1).
\end{equation}
Now, by Lemmas \ref{n-l-1}, \ref{n-l-2}, and \ref{n-l-1}, it follows that $II_1$ has the right mapping property in the space variable. 
More precisely, the claim follows from the fact that
\begin{equation}
\left\lVert R^{-3/2}\widetilde\epsilon^2\right\rVert_{H_\rho^\alpha}\lesssim \left\lVert\widetilde\epsilon\right\rVert^2_{H_\rho^{\alpha+1/2}},
\end{equation}
that, as an operator,
\begin{equation}
\left\lVert R^{-1}\widetilde\epsilon^2\right\rVert_{H_\rho^\alpha\to H_\rho^\alpha}\lesssim
\left\lVert\widetilde\epsilon\right\rVert^2_{H_\rho^{\alpha+1/2}},
\end{equation}
and from Lemma \ref{n-l-1}.
The $\tau$ behavior follows from the fact that $II_1$ has no linear term in $\epsilon$, only higher powers.

Finally, note that
\begin{equation}
R^{-3/2}f^{(2l+1)}(u_{2k-1})(R^{-1/2}\widetilde\epsilon)^{2l+1}\in (R^{-1}\widetilde\epsilon^2)^{l-1}R^{-3}\widetilde\epsilon^3 IS^0(1).
\end{equation}
After noticing that
\begin{equation}
\left\lVert R^{-3}\widetilde\epsilon^3\right\rVert_{H_\rho^\alpha}\lesssim
\left\lVert R^{-3/2}\widetilde\epsilon^2\right\rVert_{H_\rho^{\alpha+1/4}}\left\lVert\widetilde\epsilon\right\rVert_{H_\rho^{\alpha+1/2}}\lesssim
\left\lVert\widetilde\epsilon\right\rVert^3_{H_\rho^{\alpha+1/2}},
\end{equation}
the argument that $II_2$ has the right mapping property is the same as the one above for $II_1$.

\section{The Conclusion of the Argument}\label{end}

To compare the Sobolev spaces $H_\rho^\alpha$ with the usual ones $H^\beta(\R^2)$, define a map
\begin{equation}
u(R)\to (Tu)(R,\theta)=e^{i\theta}R^{-1/2}u(R).
\end{equation}
This is easily seen to be an isometry $L^2(\R^+)\to L^2(\R^2)$.
\begin{lem}\label{lemma-10}
For any $\alpha\geq0$
\begin{equation}
\left\lVert u\right\rVert_{H_\rho^{\alpha/2}(\R^+)}\asymp  \left\lVert Tu\right\rVert_{H^\alpha(\R^2)}
\end{equation}
in the sense that if one side is finite then the other is also finite and they have comparable sizes.
\end{lem}

\begin{proof}
The spaces $\hr{\beta}(\R^+)$ are defined using fractional powers of the operator $\el$, but since $\el-\el_0$ is bounded in $L^2$ and in any $\hr{\beta}$, these spaces could be defined using $\el_0$ instead. The lemma follows from the identity
\begin{equation}
\triangle Tu=T\el_0u,
\end{equation}
which holds whenever $u\in L^2$ and $\el_0 u\in L^2$.
\end{proof}

Fix now a $\nu>1/2$, and an index $k$ sufficiently large (depending on $\nu$). So far $u_{2k-1}$ and $e_{2k-1}$ have only been defined inside the cone 
$\{r\leq t\}$. They can be extended to be supported in the cone $\{r\leq2t\}$ so that they have the same regularity and all relevant derivatives match on the boundary of the light-cone. Finally, choose $\alpha$ so that
\begin{equation}
\frac{1}{4}<\alpha<\frac{\nu}{2}.
\end{equation}

The error $e_{2k-1}$ has a singularity of the type $(1-a)^{\nu-1/2}\log^m(1-a)$ on the cone $a=1$, which means that $e_{2k-1}$ is in $H^\beta$ localy around $r=t$, as long as $\beta<\nu$. On the other hand, since $e_{2k-1}$ has order one at $R=0$, 
\begin{equation}
T(R^{1/2}e_{2k-1})=e^{i\theta}R(c_0(\tau)+c_1(\tau)R^2+c_2(\tau)R^4+\cdots),
\end{equation} 
which is smooth around $R=0$. Finally, taking into account the size of the error
\begin{equation}
e_{2k-1}=\oh\left(\frac{R(\log(2+R^2))^{2k-1}}{t^2(t\lambda)^{2k}}\right),
\end{equation}
it follows that for all $\alpha<\nu/2$,
\begin{equation}
\nhr{\lambda^{-2}R^{1/2}e_{2k-1}(t(\tau),\lambda^{-1}R)}{\alpha}\lesssim\tau^{-2k+2}.
\end{equation}
Using the Propositions \ref{proposition-K-mapping}, \ref{p-transport-1}, and \ref{nonlinear-terms}, equation \eref{equation-x} can be solved through a contraction principle argument with respect to the norm
\begin{equation}
\nllr{x}{N-2}{\alpha+1/2}+\nllr{\left(\partial_\tau-2\frac{\lambda_\tau}{\lambda}\xi\partial_\xi\right)x}{N-1}{\alpha}.
\end{equation}
The transference identity and Proposition \ref{proposition-K-mapping} give that $\widetilde\epsilon=\mathcal{F}^{-1}x$ satisfies
\begin{equation}
\nhr{\widetilde\epsilon(\tau)}{\alpha+1/2}\lesssim\tau^{2-N},\quad
\nhr{\left(\partial_\tau+\frac{\lambda_\tau}{\lambda}R\partial_R\right)\widetilde\epsilon(\tau)}{\alpha}\lesssim\tau^{1-N},\quad N\leq2k.
\end{equation}
By Lemma \ref{lemma-10}, this construction yields a solution to \eref{wm} on the cone $r\leq t$, $0<t<t_0$, which is of class $H^{1+\nu-}$ on the closure of the cone. To obtain a solution on all $(0,t_0)\times\R^2$ extend the data $u(t_0,\cdot)$, $\partial_tu(t_0,\cdot)$ to all of $\R^2$ with the same smoothness. The corresponding solution to \eref{wm} will coincide with the one constructed on the cone due to the finite speed of propagation. This solution satisfies the properties stated in Theorem \ref{the-theorem}.


\appendix

\section{An Analysis of the operator $\mathcal{L}$}\label{appendix-A}

The material bellow mostly parallels section 5 of \cite{KST}. The one main difference arises from the fact that a fundamental basis of solutions for 
$\el$ is not explicitly known here. This calls for a few changes in the proof of Proposition \ref{phi-expansion-proposition}, which coresponds to Proposition 5.4 in the reference. The asymptotic expansion for $\theta$ in that same Proposition has been omitted here, as it is not needed for the main result. The rest is close to identical to \cite{KST}.

Consider the operator on $L^2(0,\infty)$
\begin{equation}
\el=\el_0+V(r)=-\partial_r^2+\frac{3}{4r^2}+V(r);\quad\quad V(r)=-\frac{1}{r^2}\left[1-f'(Q(r))\right].
\end{equation}
As $r\to 0^+$, $V(r)\sim 1$ (i.e. $\lim_{r\downarrow0}V(r)=C$ for some real number $C$). As $r\to\infty$, $V(r)\sim 1/r^4$. 
Therefore both $\el$ and $\el_0$ are self-adjoint with domains
\begin{equation}
\dom(\el)=\dom(\el_0)=\left\{\varphi\in L^2(0,\infty):\varphi,\varphi'\in AC_{loc}, \el_0\varphi\in L^2\right\}.
\end{equation}

Notice that 
\begin{equation}
\el\phi_0=0,\quad\quad\phi_0(r)=r^{3/2}Q'(r).
\end{equation}
From (\ref{hm}) it follows that $\phi_0$ is positive. As $r\to 0^+$, $\phi_0(r)\sim r^{3/2}$, as $r\to\infty$, $\phi_0(r)\sim r^{-1/2}$, so $\phi_0\not\in L^2(0,\infty)$. This two remarks together with the Sturm oscillation theorem give:

\begin{lem}\label{lel1}
The spectrum of $\el$ is purely absolutely continuous and $\mathrm{spec}\;(\el)=[0,\infty)$.
\end{lem}

To find a function $\theta_0$ such that together with $\phi_0$ it forms a fundamental system of $\el$, it is enough to ask that it satisfies 
$W(\phi_0,\theta_0)=1$. This, together with the initial condition $\theta_0(1)=0$, yields:
\begin{equation}
\theta_0(r)=-\phi_0(r)\int_1^r\frac{1}{\phi_0^2(s)}\;ds.
\end{equation}

\begin{lem}\label{lel2}
$\el$ has a fundamental system of solutions $\phi_0$, $\theta_0$ with the following properties:
\begin{itemize}
\item[i)] $\displaystyle\phi_0(r)=r^{3/2}Q'(r)$;
\item[ii)] $\displaystyle\phi_0(r)\sim r^{3/2}$, $\displaystyle\theta_0(r)\sim r^{-1/2}$ as $r\to0^+$;
\item[iii)] $\displaystyle\phi_0(r)\sim r^{-1/2}$, $\displaystyle\theta_0(r)\sim r^{3/2}$ as $r\to\infty$;
\item[iv)] $\displaystyle r^{-3/2}\phi_0(r)$ is real-analytic, $\displaystyle\phi_0(R)\in R^{3/2}S(R^{-2})$; 
\item[v)] $\displaystyle \theta_0(R)\in R^{3/2}S(R^0)$.
\end{itemize}
\end{lem}

\begin{proof}
Only the last statement needs a proof. For large $s$, $\phi_0(s)^{-2}$ admits an absolutely convergent expansion of the form:
\begin{equation}
\frac{1}{\phi_0(s)^2}=\sum_{k=1}^\infty\overline\phi_k s^{3-2k}.
\end{equation}
Therefore, for large $R$, 
\begin{equation}
\theta_0(R)\in R^{3/2}\left[\overline{\overline\phi}_0+\overline{\overline\phi}_1 R^2+\overline{\overline\phi}_2\log R+ \overline{\overline\phi}_3 R^{-2}+\cdots\right]
S(R^{-2}).
\end{equation}
\end{proof}

The following theorem will be useful:

\begin{thm}[section 3 of \cite{GZ}, Theorem 5.3 of \cite{KST}]\label{tgz} 
With the notation above:
\begin{itemize}
\item[i)] For each $z\in\C$ there exists a fundamental system $\phi(r,z)$, $\theta(r,z)$ for $\el-z$ which is analytic in $z$ for each $r>0$ and has the asymptotic behavior
\begin{equation}
\phi(r,z)\sim r^{3/2},\quad\quad\theta(r,z)\sim\frac{1}{2}r^{-1/2}\quad\text{as }r\to0^+.
\end{equation}
In particular, their Wronskian is $W(\theta(\cdot,z),\phi(\cdot, z))=1$ for all $z\in\C$. By convention, $\phi(r,z)$, $\theta(r,z)$ are real-valued for $z\in\R$.
\item[ii)] For each $z\in\C$, $\im z>0$, let $\psi^{+}(r,z)$ denote the Weyl--Titchmarsh solution of $\el-z$ at $r=\infty$ normalized so that 
\begin{equation}
\psi^+(r,z)\sim z^{-1/4}e^{iz^{1/2}}r\quad\text{as }r\to\infty,\;\im z^{1/2}>0.
\end{equation}
If $\xi>0$, then the limit $\psi^+(r,\xi+i0)$ exists point-wise for all $r>0$ and it will be denoted by $\psi^+(r,\xi)$. Moreover, define 
$\psi^-(\cdot,\xi):=\overline{\psi^+(\cdot,\xi)}$. Then $\psi^+(r,\xi)$, $\psi^-(r,\xi)$ form a fundamental system of $\el-\xi$ with asymptotic behavior
\begin{equation}
\psi^\pm(r,\xi)\sim\xi^{-1/4}e^{\pm i\xi^{1/2}r}\quad\text{as }r\to\infty.
\end{equation} 
\item[iii)] The spectral measure of $\el$ is absolutely continuous and its density is given by
\begin{equation}\label{spectral-measure-a}
\rho(\xi)=\frac{1}{\pi}\im m(\xi+i0)\chi_{[\xi>0]},
\end{equation}
with the ``generalized Weyl--Titchmarsh'' function
\begin{equation}\label{spectral-measure-b}
m(z)=\frac{W(\theta,\cdot,z),\psi^+(\cdot,z))}{W(\psi^+(\cdot,z),\phi(\cdot,z))},\quad\im z\geq0.
\end{equation}
\item[iv)] The distorted Fourier transform defined as
\begin{equation}
\mathcal{F}:f\to\widehat f(\xi)=\lim_{b\to\infty}\int_0^b\phi(r,\xi)f(r)\;dr
\end{equation}
is a unitary operator from $L^2(\R^+)$ to $L^2(\R^+,\rho)$ and its inverse is given by 
\begin{equation}
\mathcal{F}^{-1}:\widehat f\to f(r)=\lim_{\mu\to\infty}\int_0^\mu\phi(r,\xi)\widehat f(\xi)\rho(\xi)\;d\xi.
\end{equation}
Here $\lim$ refers to the corresponding $L^2$ limit.
\end{itemize}
\end{thm}

\begin{pp}\label{phi-expansion-proposition}
The $\phi(r,z)$ in Theorem \ref{tgz} admits the asolutely convergent expansion:
\begin{equation}
\phi(r,z)=\phi_0(r)+r^{-1/2}\sum_{j=1}^\infty(r^2z)^j\phi_j(r^2),
\end{equation}
where the functions $\phi_j$ are real-analytic on $[0,\infty)$ and satisfy the bounds
\begin{equation}
|\phi_j(u)|\leq\frac{C_2 C^j}{(j-1)!}\log(1+|u|),\quad|\phi_1(u)|>C\log u\;\text{if }u\gg1,
\end{equation}
where $C$, $C_2$ are positive constants.
In particular, $\phi_j(0)=0$ and $|\phi_j'(0)|\leq C_2C^j/(j-1)!$ for $j=1,2,\ldots$. 
\end{pp}

\begin{proof}
First make the ansatz
\begin{equation}
\phi(r,z)=r^{-1/2}\sum_{j=0}^\infty z^jf_j(r).
\end{equation}
The functions $f_j$ will be constructed such that the series converges in a ``reasonable'' sense. They should solve
\begin{equation}
\el(r^{-1/2}f_j)=r^{-1/2}f_{j-1},\quad\quad f_{0}(r)=r^{1/2}\phi_0(r).
\end{equation}
To obtain the $f_j$'s,  the ``forward fundamental solution'' of $\el$ is used:
\begin{equation}
H(r,s)=\frac{1}{2}\left[\phi_0(r)\theta_0(s)-\phi_0(s)\theta_0(r)\right]1_{[r>s]},
\end{equation}
therefore
\begin{equation}
f_j(r)=\frac{1}{2}\int_0^r r^{1/2}s^{-1/2}\left[\phi_0(r)\theta_0(s)-\phi_0(s)\theta_0(r)\right]f_{j-1}(s)\;ds.
\end{equation}
Remembering that $\phi_0(r)=r^{3/2}Q'(r)$ and using the notation $\chi(r)=r^2\int_1^{r}\frac{ds}{\phi_0^2(s)}$ 
(so that $\theta_0(r)=-r^{-2}\phi_0(r)\chi(r)$), the identity above becomes:
\begin{equation}
f_j(r)=\frac{1}{2}\int_0^r\frac{Q'(r)Q'(s)}{s}\left[s^2\chi(r)-r^2\chi(s)\right]f_{j-1}(s)\;ds.
\end{equation}

Note now that $\chi(r)$ can be written as:
\begin{equation}
\chi(r)=r^2\int_1^r\frac{1}{g(Q(s))^3}Q'(s)\;ds=r^2\int_{Q(1)}^{Q(r)}\frac{1}{g(\rho)^3}\;d\rho.
\end{equation}
Using the assumptions made on $g$ this gives:
\begin{align}
\chi(r)&=r^2\left[\frac{1}{2}\left(Q(1)^{-2}-Q(r)^{-2}\right)-\frac{1}{2}g'''(0)\left(\log Q(r)-\log Q(1)\right)+\cdots\right]\\
&=-\frac{1}{2}g'''(0)r^2\log r +(\text{terms analytic at 0}).
\end{align}
It follows then (by induction) that the singularity $f_j$ might have at zero is isolated and, in fact, removable. 
To see this, choose a branch of the logarithm which is holomorphic in $\C\setminus\R^-$. It is necessary to show that $f_j(r+i0)=f_j(r-i0)$ for $r<0$. Disregarding the terms not involving logarithms, it is enough to show that for any  holomorphic function $g$
\begin{equation}
\int_0^{r+i0}[\log s-\log(r+i0)]g(s)\;ds=\int_0^{r-i0}[\log s-\log(r-i0)]g(s)\;ds,
\end{equation}
which is obvious since for $s<0$
\begin{equation}
\log(s+i0)-\log(r+i0)=\log(s-i0)-\log(r-i0).
\end{equation}
Therefore, each 
$f_j$ is an even analytic function in a (uniform) neighborhood of the real line. Also, the asumption that $f_{j-1}(r)\sim r^2$ at zero implies 
$f_j(0)=0$, so $f_j(r)\sim r^2$ at zero. Induction gives that $f_j(0)=0$ for all $j$.

For the rest of this proof, let $u=r^2$, $v=s^2$, $f_j(r)=\widetilde f_j(u)$, $Q'(r)=B(u)$, $\chi(r)=X(u)$. It is easy to see that there are positive constants $C_1$, $C_2$ such that 
\begin{equation}
C_1\leq(1+u)B(u)\leq C_2,\quad\forall u.
\end{equation}
With this notation
\begin{equation}
\widetilde f_j(u)=\frac{1}{4}\int_0^u\frac{B(u)B(v)}{v}\left[vX(u)-uX(v)\right]\widetilde f_{j-1}(v)\;dv.
\end{equation}
Also
\begin{equation}
X(u)=\frac{u}{2}\int_1^u\frac{dv}{v^2B(v)^2}.
\end{equation}
Therefore, if $v\leq u$,
\begin{equation}
vX(u)-uX(v)>0,
\end{equation}
which makes (by induction and the fact that $\widetilde f_0>0$ and is increasing; see Lemma \ref{lemma-Q}) each $\widetilde f_j$ positive and increasing, 
\begin{multline}
vX(u)-uX(v)=\frac{uv}{2}\int_v^u\frac{dw}{w^2B(w)^2}
\leq Cuv\left(\frac{1}{v}-\frac{1}{u}+\log\frac{u}{v}+(u-v)\right)\\\leq Cuv\left(\frac{u-v}{uv}+\frac{u}{v}+(u-v)\right)
\leq Cu(1+u+uv),
\end{multline}
and
\begin{multline}\label{who-is-C}
\frac{B(u)B(v)}{v}\left[vX(u)-uX(v)\right]\\\leq C\frac{u}{v}[B(u)(1+u)B(v)+uB(u)vB(v)]\leq C\frac{u}{v}.
\end{multline}

Note that $\widetilde f_0(u)=uB(u)$. Then
\begin{equation}
\left|\widetilde f_1(u)\right|\leq C C_2\int_0^u\frac{u}{v}\frac{v}{1+v}\;dv=C C_2u\log(1+u).
\end{equation}
By induction, using the fact that 
\begin{equation}
\int_0^ux^{j-1}\log(1+x)\;dx\leq\frac{1}{j}u^j\log(1+u),
\end{equation}
it follows that
\begin{equation}
\left|\widetilde f_j(u)\right|\leq\frac{C_2 C^j}{(j-1)!} u^j\log(1+u),
\end{equation}
where $C$ is the same constant from the last inequality in \ref{who-is-C}.

Finally, consider
\begin{align}
\widetilde f_1(u)&=\frac{1}{4}\int_0^u\frac{B(u)B(v)}{v}\left[vX(u)-uX(v)\right]vB(v)\;dv\nonumber\\
&=\frac{1}{8}\int_0^udvB(u)B(v)^2uv\int_v^u\frac{dw}{w^2B(w)^2}.
\end{align}
Using the fact that $uB(u)$ is bounded and increasing (see Lemma \ref{lemma-Q}),
\begin{multline}
\widetilde f_1(u)\geq\frac{1}{8}uB(u)\int_0^udv \frac{v^2B(v)^2}{v}\frac{1}{u^2B(u)^2}(u-v)\\
\geq\frac{1}{8}\frac{1}{uB(u)}\int_1^uv^2B(v)^2\left(\frac{u}{v}-1\right)\;dv\geq CB(1)^2\left[u\log u-(u-1)\right].
\end{multline}
So, for $u\gg1$, 
\begin{equation}
\widetilde f_1(u)\geq Cu\log u.
\end{equation}
\end{proof}

Note that the logarithmic behavior of $\phi_1(u)$ for large $u$ is inherited by $\phi(r,\xi)$. If $1\gg\xi>0$ and $r=\delta\xi^{-1/2}$, where $\delta>0$ is a small absolute constant, then
\begin{equation}
\phi(r,\xi)\gtrsim r^{-1/2}\log r.
\end{equation}

The next proposition deals with $\psi^+$:

\begin{pp}\label{psi-expansion-proposition}
For any $\xi>0$, the solution $\psi^+(\cdot,\xi)$ from Theorem \ref{tgz} is of the form 
\begin{equation}
\psi^+(r,\xi)=\xi^{-1/4}e^{ir\xi^{1/2}}\sigma(r\xi^{1/2},r),\quad r^2\xi\gtrsim1,
\end{equation}
where $\sigma$ admits the asymptotic series approximation
\begin{equation}
\sigma(q,r)\approx\sum_{j=0}^\infty q^{-j}\psi^+_j(r),\quad\psi_0^+=1, \psi_1^+=\frac{3i}{8}+\oh\left(\frac{1}{1+r^2}\right),
\end{equation}
with zero order symbols $\psi^+_j$ that are analytic at infinity,
\begin{equation}\label{psi-j-bound}
\sup_{r>0}\left|(r\partial_r)^k\psi_j^+(r)\right|<\infty,
\end{equation}
in the sense that for all large integers $j_0$, and all idices $\alpha$, $\beta$, it holds that
\begin{equation}
\sup_{r>0}\left|(r\partial_r)^\alpha(q\partial_q)^\beta\left[\sigma(q,r)-\sum_{j=0}^{j_0} q^{-j}\psi^+_j(r)\right]\right|
\leq c_{\alpha,\beta, j_0}q^{-j_0-1}
\end{equation}
for all $q>1$.
\end{pp}

\begin{proof}

Let 
\begin{equation}
\sigma(q.r)=\xi^{1/4}\psi^+(r,\xi)e^{-ir\xi^{1/2}}.
\end{equation}
Since $\psi^+$ solves the equation
\begin{equation}
(\el-\xi)\psi^+(r,\xi)=0,
\end{equation}
it follows that $\sigma$ has to solve
\begin{equation}\label{sigma-equation}
\left(-\partial_r^2-2i\xi^{1/2}\partial_r+\frac{3}{4r^2}+V(r)\right)\sigma(r\xi^{1/2},r)=0.
\end{equation}
First look for a formal power series solution to this equation:
\begin{equation}
\sigma=\sum_{j=0}^\infty\xi^{-1/2}f_j(r),
\end{equation}
which would require that the $f_j$ satisfy
\begin{equation}\label{eq-fj}
2i\partial_rf_j=\left(-\partial_r^2+\frac{3}{4r^2}+V(r)\right)f_{j-1},\quad f_0=1.
\end{equation}
Then
\begin{equation}\label{fj-recurrence}
f_j(r)=\frac{i}{2}\partial_rf_{j-1}+\frac{i}{2}\int_r^\infty\left(\frac{3}{4s^2}+V(s)\right)f_{j-1}(s)\;ds.
\end{equation}
Inductively, it is easy to see (recalling also that $V(r)\sim r^{-4}$ at $s=\infty$) that all $f_j$ are analytic at infinity, 
with leading order term $r^{-j}$. At zero however, the $f_j$ will be singular. Using \eref{fj-recurrence} it is not hard to show, inductively, that
\begin{equation}
\left|(r\partial_r)^kf_j\right|\leq c_jr^{-j},\quad \forall k\in\mathbb{N},r>0.
\end{equation}
Indeed, suppose the claim is true for $f_{j-1}$. Then
\begin{multline}
(r\partial_r)^kf_j=\frac{i}{2}(r\partial_r)^k\partial_rf_{j-1}(r)-\frac{i}{2}(r\partial_r)^{k-1}\left(\frac{3}{4r^2}+V(r)\right)f_{j-1}(r)\\
=\frac{i}{2}\frac{1}{r}(r\partial_r)^{k+1}f_{j-1}(r)+\frac{i}{2}\left[(r\partial_r)^k,\frac{1}{r}\right](r\partial_r)f_{j-1}(r)\\
-\frac{i}{2}\left(\frac{3}{4r^2}+V(r)\right)(r\partial_r)^{k-1}f_{j-1}(r)\\-\frac{i}{2}\left[(r\partial_r)^{k-1},\frac{3}{4r^2}+V(r)\right]f_{j-1}(r).
\end{multline}
Noting that
\begin{equation}
\left[(r\partial_r)^k,\frac{1}{r}\right]=-\frac{k}{r}(r\partial_r)^{k-1},
\end{equation}
\begin{equation}
\left[(r\partial_r)^{k-1},\frac{3}{4r^2}\right]=-\frac{6(k-1)}{4r^2}(r\partial_r)^{k-2},
\end{equation}
\begin{equation}
\left[(r\partial_r)^{k-1},V(r)\right]=(k-1)rV'(r)(r\partial_r)^{k-2},
\end{equation}
the induction is complete. Let then $\psi_j^+(r)=r^jf_j(r)$. These will  satisfy \eref{psi-j-bound}.

It is known from symbol calculus that there exists a function $\sigma_{ap}(q,r)$ which satisfies
\begin{equation}
\sup_{r>0}\left|(r\partial_r)^\alpha(q\partial_q)^\beta\left[\sigma_{ap}(q,r)-\sum_{j=0}^{j_0} q^{-j}\psi^+_j(r)\right]\right|
\leq c_{\alpha,\beta, j_0}q^{-j_0-1}
\end{equation}
for all natural numbers $\alpha$, $\beta$, and $j_0$.  However, $\sigma_{ap}$ will not solve \eref{sigma-equation}. Define the error
\begin{equation}
e(r\xi^{1/2},r)=\left(-\partial_r^2-2i\xi^{1/2}\partial_r+\frac{3}{4r^2}+V(r)\right)\sigma_{ap}(r\xi^{1/2},r).
\end{equation}
It is easy to see that
\begin{equation}
\left|(r\partial_r)^\alpha(q\partial_q)^\beta e(q,r) \right|\leq c_{\alpha,\beta,j}r^{-2}q^{-j},
\end{equation}
for all $\alpha$, $\beta$, and $j$.  Let $\sigma_1=-\sigma+\sigma_{ap}$. This $\sigma_1$ has to satisfy
\begin{equation}\label{sigma-1-equation}
\left(-\partial_r^2-2i\xi^{1/2}\partial_r+\frac{3}{4r^2}+V(r)\right)\sigma_1(r\xi^{1/2},r)=e(r\xi^{1/2},r).
\end{equation}

To obtain estimates on $\sigma_1$, first define $\vec v=(v_1,v_2)=(\sigma_1,r\partial_r\sigma_1)$. The equation \eref{sigma-1-equation} can be written in terms of $\vec v$ as
\begin{equation}
\partial_r\vec v-
\left(  
\begin{array}{cc}
0&r^{-1}\\
\frac{3}{4r}+rV(r)&r^{-1}-2i\xi^{1/2}
\end{array}
\right)\vec v=\left(\begin{array}{c}0\\-re \end{array}\right).
\end{equation}
From this it follows that
\begin{equation}
\frac{d}{dr}|\vec v|^2\geq -C\left(r^{-1}|\vec v|^2+r|\vec v|\,|e| \right),
\end{equation}
so
\begin{equation}
\frac{d}{dr}|\vec v|\geq -C\left(r^{-1}|\vec v|+r|e| \right).
\end{equation}
By Gronwall's inequality,
\begin{equation}
|\vec v(r)|\leq\int_r^\infty\left(\frac{s}{r}\right)^Cs|e(s)|\;ds.
\end{equation}
For large $j$,
\begin{equation}
|e|<C\xi^{-j/2}r^{-j-2},
\end{equation}
which implies that
\begin{equation}
|\vec v|<C_j\xi^{-j/2}r^{-j}=C_jq^j,
\end{equation}
for large $j$. Entirely similar arguments can be applied to $(r\partial_r)^\alpha(q\partial_q)^\beta\vec v$, to conclude in the end that
\begin{equation}
\left|(r\partial_r)^\alpha(q\partial_q)^\beta \sigma_1(q,r)\right|\leq C_{\alpha,\beta,j}q^{-j},
\end{equation}
for large $j$ and any $\alpha$ and $\beta$. Then $\sigma=\sigma_{ap}-\sigma_1$ is as desired.
\end{proof}

The last result of this section deals with the spectral measure of $\el$.
\begin{pp}
\begin{itemize}
\item[i)] There is a function $a(\xi)$ such that
\begin{equation}
\phi(r,\xi)=a(\xi)\psi^+(r,\xi)+\overline{a(\xi)\psi^+(r,\xi)}
\end{equation}
which is smooth, always nonzero, and has size\footnote{By $a\asymp b$ it is meant that there is a positive constant $C$ such that $C^{-1}a<b<Ca$.}
\begin{equation}\label{size-of-a}
|a(\xi)|\asymp\left\{\begin{array}{cl}
-\xi^{1/2}\log\xi\quad&\xi\ll1\\
\xi^{-1/2}&\xi\gtrsim1
\end{array}\right..
\end{equation}
Moreover, this function satisfies the bounds
\begin{equation}\label{symbol-a-bounds}
\left|(\xi\partial_\xi)^ka(\xi)\right|\leq c_k|a(\xi)|,\quad\forall\xi>0.
\end{equation}
\item[ii)] The spectral measure $\rho(\xi)d\xi$ has density
\begin{equation}
\rho(\xi)=\frac{1}{\pi}|a(\xi)|^{-2}
\end{equation}
and therefore satisfies
\begin{equation}
\rho(\xi)\asymp\left\{\begin{array}{cl}
\frac{1}{\xi(\log\xi)^2}\quad&\xi\ll1\\
\xi&\xi\gtrsim1
\end{array}\right..
\end{equation}
\end{itemize}
\end{pp}

\begin{proof}
 i) Since $\phi$ is real valued and $W(\psi^+,\psi^-)=-2i$, the function $a$ must be
\begin{equation}
a(\xi)=-\frac{i}{2}W(\phi(\cdot,\xi),\psi^-(\cdot,\xi)).
\end{equation}
By Proposition \ref{phi-expansion-proposition} it follows that both $\phi(\xi^{-1/2},\xi)$ and $(r\partial_r\phi)(\xi^{-1/2},\xi)$ can be written in the form $\xi^{1/4}f(\xi^{-1})$ with $f(u)$ analytic and satisfying
\begin{equation}
|f(u)|\lesssim \log(1+|u|).
\end{equation}
By Proposition \ref{psi-expansion-proposition} it follows that both $\psi^+(\xi^{-1/2},\xi)$ and $(r\partial_r\psi^+)(\xi^{-1/2},\xi)$ can be written in the form $\xi^{-1/4}h(\xi^{-1/2})$ with $h$ satisfying the bounds
\begin{equation}
|(r\partial_r)^kh(r)|\leq c_k.
\end{equation}
Then the function $a$ is a sum of terms of the form $\xi^{1/2}f(\xi^{-1})h(\xi^{-1/2})$, with $f$ and $h$ as above. The bounds \eref{symbol-a-bounds} and
the upper bounds in \eref{size-of-a} then follow.

To prove the lower bounds, begin by noting that
\begin{equation}
\im(\psi^+(r,\xi)\partial_r\psi^-(r,\xi))=-1.
\end{equation}
Since $\phi$ is real-valued, this gives
\begin{equation}
\im\left[\partial_r\psi^+(r,\xi)W(\phi(\cdot,\xi),\psi^-(\cdot,\xi))\right]=-\partial_r\phi(r,\xi),
\end{equation}
which implies that for all $r$
\begin{equation}
|a(\xi)|\geq\frac{|\partial_r\phi(r,\xi)|}{2|\partial_r\psi^+(r,\xi)|}.
\end{equation}
There is a small constant $\delta$ such that, if $r=\delta\xi^{-1/2}$, by Proposition \ref{phi-expansion-proposition}
\begin{equation}
|\partial_r\phi(r,\xi)|\gtrsim r^{-3/2}\log(1+r^2),
\end{equation}
and by Proposition \ref{psi-expansion-proposition}
\begin{equation}
|\partial_r\psi^+(r,\xi)|\lesssim\xi^{1/4}(r^2\xi)^{-j_0}.
\end{equation}
These give the lower bounds in \eref{size-of-a}.

ii) $\psi^+$ can be written in terms of $\phi$ and $\theta$ as
\begin{equation}
\psi^+=-\phi W(\psi^+,\theta)+\theta W(\psi^+,\phi).
\end{equation}
Since both $\phi$ and $\theta$ are real-valued, inserting into $W(\psi^+,\psi^-)=-2i$, it follows that
\begin{equation}
\im\left[W(\psi^+,\theta)W(\psi^-,\phi)\right]=-1.
\end{equation}
Inserting into \eref{spectral-measure-a} and \eref{spectral-measure-b}, this yields
\begin{equation}
\rho(\xi)=\frac{1}{\pi}\frac{\im\left[W(\psi^+,\theta)W(\psi^-,\phi)\right]}{|W(\psi^+,\phi)|^2}=\frac{1}{\pi}|W(\psi^+,\phi)|^{-2}=\frac{1}{\pi|a(\xi)|^2}.
\end{equation}
\end{proof}

\bibliographystyle{plain}

\bibliography{unu}

\end{document}